\documentclass[11pt, a4paper]{amsart}

\usepackage[T1]{fontenc}
\usepackage[utf8]{inputenc}
\usepackage{lmodern}

\usepackage[american]{babel}
\usepackage[babel, final]{microtype}
\usepackage{amssymb}
\usepackage[all]{xy}
\usepackage{tabularx}
\usepackage{booktabs} 

\usepackage{lscape}

\usepackage{mathtools}

\usepackage[pdftex, dvipsnames]{xcolor}
\usepackage[pdftex, final]{graphicx}
\usepackage{import}
\usepackage{caption}
\usepackage{changepage}
\usepackage[lofdepth,lotdepth]{subfig}
\usepackage{chngcntr}

\usepackage{svg}
\usepackage{ifpdf}
\usepackage{comment}
\usepackage{rotfloat}
\usepackage{array}
\usepackage{multirow}
\usepackage{rotating}
\usepackage{longtable}

\usepackage{pinlabel}
\usepackage[pdftex,%
    a4paper,%
    includehead,%
    includefoot,%
    nomarginpar,%
    lmargin=0.8in,%
    rmargin=0.8in,%
    tmargin=1in,%
    bmargin=1in,%
]{geometry}
\usepackage[pdftex,%
    final,%
    colorlinks=true,%
    linkcolor=NavyBlue,%
    citecolor=NavyBlue,%
    filecolor=NavyBlue,%
    menucolor=NavyBlue,%
    urlcolor=NavyBlue,%
    bookmarks=true,%
    bookmarksdepth=3,%
    bookmarksnumbered=true,%
    bookmarksopen=true,%
    bookmarksopenlevel=2,%
]{hyperref}
\hypersetup{
    pdftitle={Unknotting via null-homologous twists and multi-twists},
    pdfauthor={Samantha Allen, Kenan Ince, Seungwon Kim, Benjamin Matthias Ruppik, Hannah Turner},
    pdfsubject={Unknotting via null-homologous twists and multi-twists},
    pdfkeywords={knot theory, unknotting number, untwisting number, surgery diagram, Kirby moves}
}
\usepackage{cleveref}

\newcommand{\vertleq}{\rotatebox{270}{$\leq$}}

\usepackage{csquotes}
\usepackage[
	backend=biber,
	style=alphabetic,
	sorting=nyt,
	giveninits=true,
	maxnames=10,
	url=false,
	backref=true
]{biblatex}
\DefineBibliographyStrings{english}{
  backrefpage={$\uparrow$},
  backrefpages={$\uparrow$}
}

\DeclareFieldFormat[article,incollection,inproceedings]{title}{#1}
\DeclareFieldFormat[unpublished]{title}{\textit{#1}}

\usepackage{tikz}
\usetikzlibrary{tikzmark}
\usepackage{tikz-cd}

\graphicspath{{pictures/}}

\input{commands}

\usepackage{todonotes}

\theoremstyle{remark}
\newtheorem*{note}{Note}

\title{Unknotting via null-homologous twists and multi-twists}

\author[Allen]{Samantha Allen}
\address{Duquesne University, Pittsburgh, PA, USA}
\email{\href{mailto:allens6@duq.edu}{allens6@duq.edu}}
\urladdr{\url{https://samantha-allen.github.io/}}

\author[\.{I}nce]{Kenan \.{I}nce}
\address{Westminster College, Salt Lake City, Utah, USA}
\email{\href{mailto:kince@westminstercollege.edu}{kince@westminstercollege.edu}}
\urladdr{\url{https://cs.westminstercollege.edu/~kince}}

\author[Kim]{Seungwon Kim}
\address{Sungkyunkwan University}
\email{\href{mailto:seungwon.kim@skku.edu}{seungwon.kim@skku.edu}}
\urladdr{\url{https://sites.google.com/view/seungwonkim}}

\author[Ruppik]{Benjamin Matthias Ruppik}
\address{Heinrich-Heine-Universit\"{a}t D\"{u}sseldorf}
\email{\href{mailto:benjamin.ruppik@hhu.de}{benjamin.ruppik@hhu.de}}
\urladdr{\url{https://ben300694.github.io}}

\author[Turner]{Hannah Turner}
\address{Georgia Institute of Technology}
\email{\href{mailto:hannah.turner@math.gatech.edu}{hannah.turner@math.gatech.edu}}
\urladdr{\url{https://sites.google.com/view/hturner/}}
\keywords{
	4-manifolds,
	surgery diagram,
	unknotting operation,
	untwisting number.
}
\def\subjclassname{\textup{2020} Mathematics Subject Classification}
\expandafter\let\csname subjclassname@1991\endcsname=\subjclassname
\expandafter\let\csname subjclassname@2000\endcsname=\subjclassname
\subjclass{
    57K40, 
    57K10. 
    \hfill
}

\begin{document}
\begin{abstract}
    The \textit{untwisting number} of a knot $K$ is the minimum number of \textit{null-homologous twists} required to convert $K$ to the unknot.  Such a twist can be viewed as a generalization of a crossing change, since a classical crossing change can be effected by a null-homologous twist on 2 strands.  While the unknotting number gives an upper bound on the smooth 4-genus, the untwisting number gives an upper bound on the topological 4-genus.
    The \textit{surgery description number}, which allows multiple null-homologous twists in a single twisting region to count as one operation, lies between the topological 4-genus and the untwisting number.
    We show that the untwisting and surgery description numbers are different for infinitely many knots, though we also find that the untwisting number is at most twice the surgery description number plus $1$.
\end{abstract}

\maketitle

\section{Introduction}

Given two knot or link diagrams $D_1,D_2$ of links $L_1,L_2$ which differ only inside small disks $\Delta\subset D_1,\Delta^\prime\subset D_2$ containing at least one crossing, a \textit{local move on $L_1$} is the act of replacing $\Delta$ with $\Delta^\prime$, hence converting $D_1$ to $D_2$. An \textit{unknotting operation} is a local move such that, for any diagram $D$ of a knot $K$, we may transform $D$ into a diagram of the unknot via a finite sequence of these local moves. A natural question in knot theory is: given an unknotting operation and a knot $K$, how many such operations are needed to turn $K$ into the unknot? The most common such unknotting operation is a crossing change, which gives rise to the unknotting number $u(K)$. We study additional unknotting operations below which are organized in Table \ref{table:definitions_of_invariants}.

The unknotting number is quite simple to define but has proven
difficult to compute. Fifty years ago, Milnor conjectured that the
unknotting number for the $(p,q)$-torus knot was $(p-1)(q-1)/2$. It was not until 1993 that Kronheimer and Mrowka proved this conjecture in two celebrated papers \cite{kronheimer_gauge_1993,kronheimer_gauge_1995}.  

In \cite{mathieu_chirurgies_1988-1}, Mathieu and Domergue defined
another generalization of unknotting number. In \cite{livingston_slicing_2002},
Livingston worked with this definition. He described it as follows: 
\begin{quotation}
``One can think of performing a crossing change as grabbing two parallel
strands of a knot with opposite orientation and giving them one full
twist. More generally, one can grab $2k$ parallel strands of $K$
with $k$ of the strands oriented in each direction and give them
one full twist.''
\end{quotation}
We call such a twist a \emph{null-homologous twist}. It is described in \cite{ince_untwisting_2016} how a crossing change may be encoded as a $\pm1$-surgery on a null-homologous unknot
$U\subset S^{3}-K$ bounding a disk $D$ such that $D\cap K=2$ points.
From this perspective, a generalized crossing change is a relaxing
of the previous definition to allow $D\cap K=2k$ points for any $k$,
provided $\text{lk}(K,U)=0$ (see Fig.\ \ref{nulltwist}).
In particular, null-homologous twists are an unknotting operation.

One may then naturally define the \emph{untwisting number} $\tu(K)$
to be the minimum length, taken over all diagrams of $K$, of a sequence
of generalized crossing changes beginning at $K$ and resulting in
the unknot. This has been the subject of much research in recent years \cite{MR4171377,ince_untwisting_2016,inceUntwistingInformationHeegaard2017a,livingston_null-homologous_2019-1,mccoy2021gaps,mccoy2019nullhomologous}.

\begin{figure}[ht]
    \begin{tikzpicture}
    	\node[anchor=south west,inner sep=0] (image) at (0,0) {\includegraphics[]{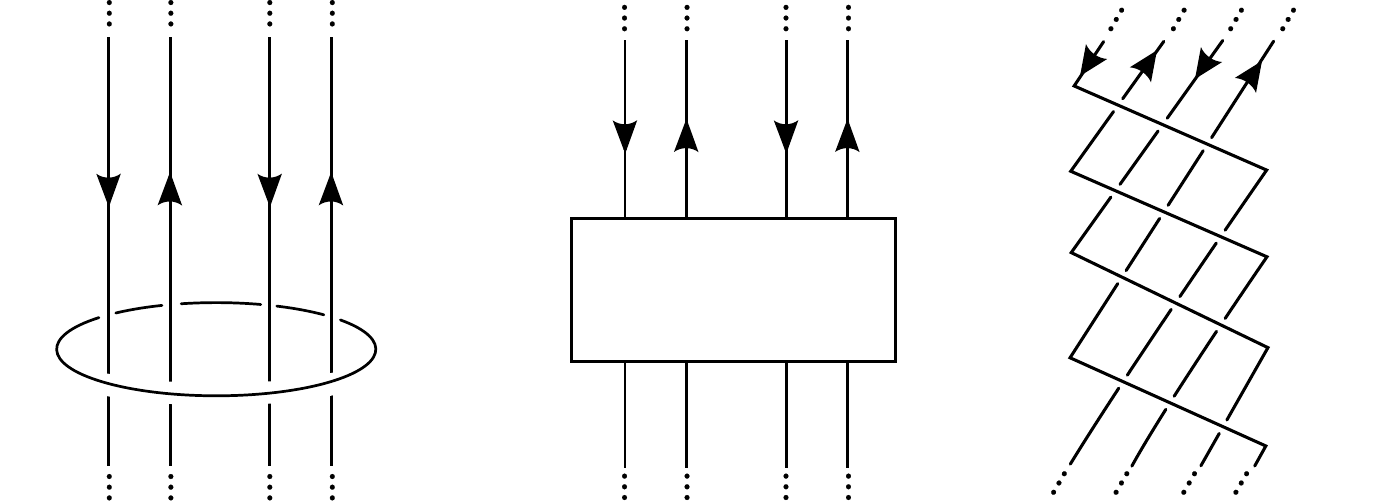}};
    	\begin{scope}[x={(image.south east)},y={(image.north west)}]
            \node at (0,0.3) {\large{$+1$}};
            \node at (0.53, 0.41){ \large{$1$ \text{full LH twist}}};
            \node at (.33,0.5) {\large{$=$}};
            \node at (0.71,0.5) {\large{$=$}};
          
        \end{scope}
    \end{tikzpicture}

    \protect\caption{\label{nulltwist}A left-handed null-homologous twist on 4 strands. 
    }
\end{figure}
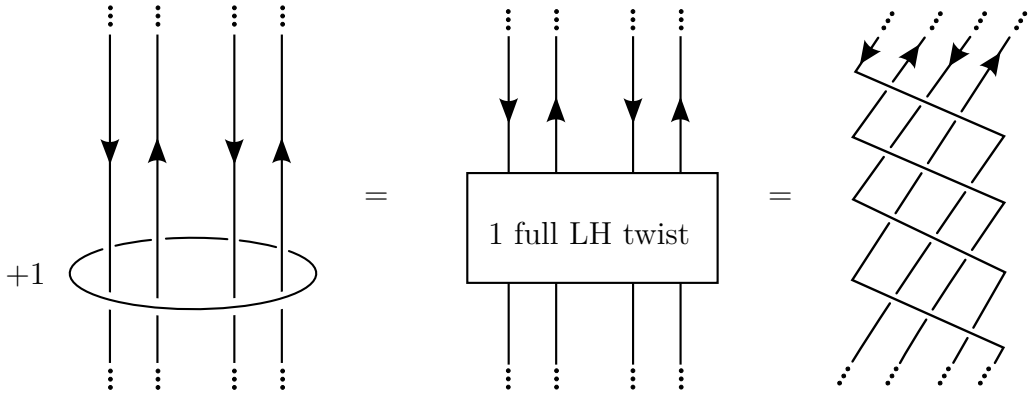

There are many variations of the unknotting number and untwisting number, some of which we collect in \Cref{table:definitions_of_invariants}.
One variant we will study, due to Nakanishi \cite{nakanishi2005} (and called the ``surgical description number'' in that paper), is what we and many other authors call the \emph{surgery description number} $\sd(K)$ of a knot. Again we consider null-homologous twists but now allow any number of full twists to be added in the twisting region; we may call this a null-homologous $m$-twist for $m\in\mathbb{Z}$ to specify the number of twists (with sign) being effected. Then $\sd(K)$ is the minimal number of $m$-twists necessary to unknot $K$.
(Here, the value of $m$ may change from move to move.)   

Another natural variant (due to Murakami \cite{murakami_algebraic_1990}) is the \emph{algebraic unknotting number} $u_{a}(K)$, the minimum number of crossing changes necessary to turn a given knot into an Alexander polynomial-one knot. 
Freedman \cite{freedman1982} showed that knots with Alexander polynomial equal to one are topologically slice (in other words, with topological 4-genus $g_4^{\topological} = 0$); topologically slice knots are indistinguishable from the unknot by \textit{classical invariants}, or knot invariants derived from the Seifert matrix. We consider the similarly defined \emph{algebraic untwisting number} $\tu_a(K)$ and \emph{algebraic surgery description number} $\sd_a(K)$, measuring the number of null-homologous twists or $m$-twists, respectively, needed to obtain a knot with Alexander polynomial $1$, as well. 

A tight classical upper bound on the topological $4$-genus $g_4^{\topological}$ of a knot is the \textit{algebraic genus} $\alggenus$  defined in \cite{feller_classical_2019}. 
Distinguishing the algebraic genus from other upper bounds on $g_{4}^{\topological}$, such as the algebraic unknotting number, is often achieved by using the bound $\alggenus\leq g_3$, where $g_{3}(K)$ is the $3$-genus of $K$.  In Section \ref{g_alg<sd_a}, we provide the first (to our knowledge) known infinite family of knots $L_n$ for which $\alggenus(L_n) < u_a(L_n)$ for all $n\in\mathbb{N}$, and since the $3$-genus of our examples is large, we do so without using $g_3$. 

The untwisting number connects to recent work of Manolescu and Piccirillo \cite{manolescu_zero_2021} on candidates for exotic definite $4$-manifolds, which uses the concept of \textit{strong $H$-sliceness} in definite connected sums of $\pm\mathbb{CP}^{2}$. (See Section \ref{g_alg<sd_a} for a related definition.)
It follows from Proposition $4.1$ of \cite{inceUntwistingInformationHeegaard2017a} that, if $K$ can be unknotted using $n$ positive (respectively, negative) nullhomologous twists, then $K$ is strongly topologically $H$-slice in $X:= B^{4}\#^{n}\mp\mathbb{CP}^{2}\cong\#^n\mp\mathbb{CP}^{2}$.
We use this fact to obstruct knots from having $\sd_a=1$ in Section \ref{g_alg<sd_a}.

\subsection*{Results}

Our main results involve various relationships between the untwisting number and the surgery description number.  To start, we give the first known examples (to the authors' knowledge) such that $\sd \neq \tu$. See Section \ref{sec:sd-tu} for a description of these knots.
\begin{theorem}
    \label{thm:sd != tu} There are infinitely many knots $\{K_n\}$ with $\sd(K_n)=1$ and $\tu(K_n)=2$.
\end{theorem}
\noindent This, of course, leads to questions about how far apart the surgery description number and the untwisting number can be.  

\begin{question}
     Can $\tu$ and $\sd$ be arbitrarily far apart?
\end{question}

Answering such a question is made more difficult by the close relationships between $\tu$ and $\sd$ both in definition and in values, demonstrated by the following two results.
\begin{theorem}
    \label{thm:algebraic}
    Let $K\subset S^3$ be a knot.
    Then $\sd_a(K) \leq \tu_a(K) \leq 2\sd_a(K)$.
\end{theorem}

\begin{theorem} 
    \label{thm:geometric}
    Let $K\subset S^3$ be a knot.
    Then $\sd(K) \leq \tu(K) \leq 2\sd(K)+1$.
\end{theorem}

\noindent The proof of \Cref{thm:algebraic} relies on work of Duncan McCoy relating the untwisting number to the algebraic genus \cite{mccoy2019nullhomologous}.
The proof of Theorem \ref{thm:geometric} is constructive (involving surgery diagrams and Kirby calculus) and allows one to reduce multiple twists in a single region to at most 3 twists in separate regions.

\subsection*{Organization}

In Section \ref{sec:alg untwisting invts}, we give formal definitions of all relevant invariants, as well as some useful prior results.  We also prove Theorem \ref{thm:algebraic} as a consequence 
of the work of McCoy in \cite{mccoy2019nullhomologous}.
We prove Theorem \ref{thm:sd != tu} in Section \ref{sec:sd-tu} by providing an infinite family of examples where the invariants disagree.
Theorem \ref{thm:geometric} is proved in Section \ref{sec:geometric}.

\begin{table}[h!]
	\centering
	\small
	\begin{tabularx}{\textwidth}{lX}
		\toprule
		Invariant & Definition \\
		\midrule
		$u(K)$ & unknotting number of $K$, i.e.\, minimal number of crossing changes to unknot \\
		$u_{a}(K)$ & alg.\ unknotting number, minimal number of crossing changes to Alexander polynomial-one knot \\
        $\tu_a(K)$ & alg.\ untwisting number, minimal number of null-homologous twists to Alexander polynomial-one knot \\ 
		$\tu(K)$ & untwisting number, i.e., the minimal number of null-homologous twists to unknot \\
		$\sd(K)$ & surgery description number, i.e., the minimal number of null-homologous multi-twists (on the \newline same region of $K$) to unknot \\
		$\sd_{a}(K)$ & algebraic surgery description number, i.e., the minimal number of null-homologous multi-twists \newline (on the same region of $K$) to Alexander polynomial-one knot\\
		$\alggenus(K)$ & algebraic genus, i.e., minimal difference in genus between a Seifert surface $F$ for $K$ and a subsurface \newline whose boundary is an Alexander polynomial-one knot\\
		\bottomrule
	\end{tabularx} 
	\caption{Overview of knot invariants appearing in this paper.}
	\label{table:definitions_of_invariants}
\end{table}
\section*{Acknowledgements}

This work is the product of a research group formed during the American Institute for Mathematics (AIM) virtual Research Community on \emph{4-dimensional topology}.
The paper is based on work completed while the first, second, third, and fifth authors were in residence at the Mathematical Sciences Research Institute in Berkeley, California, during the summer of $2022$.
SK was supported by a National Research Foundation of Korea (NRF) grant funded by the Korea government (MSIT) (NRF-2022R1C1C2004559).
BR was supported by the Max Planck Institute for Mathematics in Bonn.

We want to thank Peter Feller, Stefan Friedl and Duncan McCoy for their valuable comments on an earlier draft.  In particular, Duncan McCoy suggested a strengthening of Theorem \ref{thm:geometric}.

\section{Algebraic untwisting invariants} 
\label{sec:alg untwisting invts}

One way to study an unknotting operation is to analyze its impact on the Alexander polynomial of a knot. The effect of an operation on the Alexander polynomial gives rise to \textit{algebraic unknotting operations}:

\begin{definition}
    Given an unknotting operation $\mathcal{U}$ and a knot $K$, the \textit{algebraic $\mathcal{U}$-number} $\mathcal{U}_{a}(K)$ is the minimal number of  $\mathcal{U}$-operations that must be performed in order to convert $K$ into a knot with Alexander polynomial $1$.
\end{definition}

We certainly have that $\mathcal{U}_{a}(K)\leq\mathcal{U}(K)$ for any unknotting operation  $\mathcal{U}$ and knot $K$. A lower bound on the algebraic unknotting and untwisting numbers is the topological $4$-genus. Another (typically tighter) upper bound on the topological $4$-genus is the \textit{algebraic genus}, defined by Feller and Lewark \cite{feller_classical_2019}.

\begin{definition}
    \label{algebraic-genus-defintion}
    The \textit{algebraic genus} $\alggenus(K)$ of a knot $K$ is the minimum difference in genus $g(F)-g(F^\prime)$ between a Seifert surface $F$ for $K$ and a subsurface $F^\prime\subset F$ with the property that $\partial F^\prime=K^\prime$ is a knot with $\Delta_{K^\prime}(t)=1$.
\end{definition}

We note that Definition \ref{algebraic-genus-defintion} implies that a knot $K$ has $\alggenus(K)=0$ if and only if $\Delta_K(t)=1$. 
McCoy proves the following useful characterization of the sensitivity of the algebraic genus to null-homologous twisting.

\begin{theorem}[{\cite[Theorem 1.1]{mccoy2019nullhomologous}}]
    \label{twist-algebraic-genus}
    If $K$ and $K^\prime$ are knots and $m,n \in \mathbb{Z}$ are such that a null-homologous $m$-twist followed by a null-homologous $n$-twist on $K$ results in $K^\prime$ and $-mn$ is a square, then
    \[
        |\alggenus(K)-\alggenus(K^\prime)| \leq 1.
    \]
\end{theorem}

\begin{proposition}[{\cite[Proposition 10]{mccoy2019nullhomologous}}]
    \label{g-alg-staircase}
    Given a link with $\alggenus(K)>0$, there exists a knot $K^\prime$ with $\alggenus(K^\prime)=\alggenus(K)-1$ such that $K$ can be obtained from $K^\prime$ by one right-handed and one left-handed null-homologous twist.
\end{proposition}

Feller and Lewark show that for a knot $K$ the algebraic genus and the algebraic unknotting number are related by $\alggenus(K)\leq u_a(K)\leq 2\alggenus(K)$ \cite{feller_classical_2019}. We will show that in fact 
\begin{equation}
    \label{g_alg-chain} 
    \alggenus(K)\leq \sd_a(K)\leq u_a(K)\leq 2\alggenus(K)\leq 2\sd_a(K)
\end{equation}

and that $\sd_a(K)$ can provide a better lower bound for $u_a(K)$ than $\alggenus(K)$. 
We first show the following, that the algebraic genus is in fact a lower bound on the algebraic surgery description number.

\begin{proposition}
    \label{g_alg-sd_a}
    Let $K\subset S^{3}$ be a knot.
    Then $\alggenus(K)\leq \sd_{a}(K)$.
\end{proposition}

\begin{proof}
    Suppose that $K$ is a knot with $\sd_a(K)=k$. Then there exists a sequence of $k$ null-homologous $m_i$-twists (for $1\leq i\leq k$) converting $K$ to the unknot (which in particular has algebraic genus 0).
    By Theorem \ref{twist-algebraic-genus} (with $n=0$), each of these $m_i$-twists decreases the algebraic genus by at most $1$, whence $\alggenus(K)\leq k$.
\end{proof}

Before proving \Cref{thm:algebraic}, we need to note the following result of \.{I}nce.
\begin{theorem}[\cite{ince_untwisting_2016}]
\label{thm:ua=tua}
Let $K\subset S^3$ be a knot.  Then $u_a(K) = \tu_a(K)$.
\end{theorem}
We are now ready to prove \Cref{thm:algebraic}: that, for any knot $K$,  $\sd_a(K)\leq \tu_{a}(K) \leq 2\sd_a(K).$

\begin{proof}[Proof of Theorem \ref{thm:algebraic}]
    Since any crossing change can be realized as a null-homologous twist, and a single null-homologous twist is an $m-$twist with $m=\pm 1$, we have $\sd_{a}(K)\leq \tu_a(K)$ for any knot $K$. Theorem \ref{thm:ua=tua} then implies that $\sd_{a}(K) \leq u_a(K)$ for any knot $K$. Combining Proposition \ref{g_alg-sd_a} with Feller-Lewark's result \cite{feller_classical_2019} that $u_a(K)\leq 2\alggenus(K)$, we have that $u_a(K)\leq 2\sd_a(K).$
\end{proof}

\begin{note}
    In \cite{borodzik_untwisting_2019-1}, Borodzik showed that the minimal number of null-homologous twists \textit{on two strands} needed to convert a knot $K$ into a knot with Alexander polynomial $1$ is always less than \textit{three} times the algebraic surgery description number. 
    In fact, our Theorem \ref{thm:algebraic}, together with the fact that a crossing change is a special case of a null-homologous two-strand twist and the fact that $u_a = \tu_a$, refines this upper bound to twice the algebraic surgery description number. 
   
\end{note}

Even though the algebraic unknotting $u_a(K)$ and untwisting numbers $\tu_a(K)$ coincide, the algebraic surgery description number $\sd_a(K)$ can be strictly less than $u_a(K)=\tu_a(K)$; this is the content of Corollary \ref{cor:infinite-sda-neq-tua}.

To conclude that the algebraic surgery description number $sd_a(K)$ can be a better lower bound on the algebraic unknotting number $u_a(K)$ than the algebraic genus $\alggenus(K)$, we should show that there is a knot for which $\alggenus(K)\neq sd_a(K).$ We provide infinitely many examples with this property in the next section.

\section{Infinite families of knots with \texorpdfstring{$\alggenus<\sd_a$}{algebraic genus less than algebraic surgery description number}}
\label{g_alg<sd_a}

A knot $K \subset S^3$ is called \emph{topologically $H$-slice} in a closed, smooth $4$--manifold $M$ if $K\subset \partial (M\setminus B^4)$ bounds a locally flat, properly embedded, null-homologous topological disk in $M\setminus B^4$. In the context of this paper, if a knot $K$ can be converted to a knot which is topologically slice in $B^4$ via only left-handed or, respectively, only right-handed nullhomologous  $m$--twists, then $K$ is topologically $H$--slice in $\#_{n}\pm\mathbb{CP}^{2}$ for some $n$.
The following well-known result (see, for example, \cite{kauffman_signature_1976}) will be useful to us.

\begin{proposition}
    \label{prop: H-slice sig}
    If a knot $K$ is topologically $H$--slice in $\#_{m}\mathbb{CP}^{2}$, then for any $\omega \in S^1$ with $\Delta_K(\omega) \neq 0$, 
    $$-2m\leq \sigma_K(\omega) \leq 0,$$
    where $\sigma_K(\omega)$ is the Levine-Tristram signature function of $K$.
\end{proposition}

\noindent In particular, the proposition above implies that if the signature function of a knot takes on both positive and negative values, then $\sd_a \neq 1$.

\begin{theorem} 
    \label{thm: sd>g_alg}
    If $K$ and $K^\prime$ are knots such that 
    \begin{itemize} 
        \item $u_a(K) = u_a(K^\prime) = 1$, 
        \item the signature function of $K$ takes a positive value, and
        \item the signature function of $K^\prime$ takes a negative value,
    \end{itemize}
    then $\alggenus(K\#K^\prime) = 1$.  If, in addition, the signature function of $K\#K^\prime$ takes both positive and negative values, then $\alggenus(K\#K^\prime)< \sd_a(K\#K^\prime)$.  
\end{theorem}

\begin{proof}
    Suppose $K$ and $K^\prime$ are knots which satisfy the bulleted assumptions of Theorem \ref{thm: sd>g_alg}.   Now consider the knot $J = K\#K^\prime$. Note that because $K$ and $K^\prime$ have nontrivial signature functions, neither $K$ nor $K^\prime$ has Alexander polynomial 1.  So $\Delta_J(t) \neq 1$ and $\alggenus(J) \neq 0$.
    Because $u_a(K) = u_a(K^\prime) = 1,$ the knots $K$ and $K^\prime$ can be converted into knots with Alexander polynomial 1 via a single crossing change. Recall that a crossing change can change the Levine-Tristram signature function by at most $\pm 2$, where the sign depends on the sign of the crossing change (see, for example, \cite[Lemma 5]{livingston2018}).  This implies that the crossing changes converting $K$ and $K^\prime$ to Alexander polynomial 1 knots can be taken to be of opposite signs.  So the knot $J$ can be converted into a knot $J^\prime$ with $\Delta_{J^\prime}(t) = 1$ via a sequence of two crossing changes, one positive and one negative. Since a crossing change can be realized by single null-homologous twist, by Proposition \ref{twist-algebraic-genus} we have that
    $$|\alggenus(J)-\alggenus(J^\prime)|\leq 1.$$
    Because $\Delta_{J^\prime}(t) = 1,$ we have that $\alggenus(J^\prime)=0$.  So $\alggenus(J) = 1$ as desired.
    
    Now, suppose that $K$ and $K^\prime$ also satisfy that $\sigma_{J}(\omega) = \sigma_{K}(\omega)+\sigma_{K^\prime(\omega)}$ takes both positive and negative values.  By Proposition \ref{prop: H-slice sig}, $J=K\#K^\prime$ is not topologically $H$-slice in $\#_{m}\pm\mathbb{CP}^{2}$ for any $m\in\mathbb{N}$.  
    In particular, $J$ cannot be converted to a topologically slice knot using a single $\mathrm{sd}$--move.  Thus $\mathrm{sd}_a(J) > 1.$
\end{proof}

\begin{theorem}
    There exists an infinite family $\{K_n\}_{n=2}^{\infty}$ of prime knots such that $\alggenus(K_n)<\sd_a(K_n)$ for all $n\geq 2$.
\end{theorem}

\begin{note}
In fact, since the Tristram-Levine signature and algebraic unknotting number of a knot $K$ are invariants of the $S$-equivalence class of its Seifert matrix, for any Seifert matrices $V,V^\prime$ satisfying the conditions of Theorem \ref{thm: sd>g_alg}, there exist infinitely many knots $K,K^\prime$ with Seifert matrices in the $S$-equivalence classes of $V,V^\prime$, respectively, satisfying the conclusions of the theorem. In particular, our $K_n$ can be chosen to have any adjective (e.g. hyperbolic, quasipositive,...) for which there are infinitely many representative knots with that property in each S-equivalence class, since our proof relies only on the S-equivalence class of $K_n$. 
\end{note}

In the proof below, we exhibit a concrete family of prime knots via cabling because cabling seems to be of independent interest.

\begin{proof} 
Suppose that $K$ and $K'$ are knots which satisfy all the assumptions of Theorem \ref{thm: sd>g_alg}.
    For example, we can take $K = 10_{32}$ and $K' = -10_{82}$ (see Figure \ref{fig: sd_a>g_alg}). For $n\geq 2$, let $K_n := (K\#K')_{n,1}$, the $(n,1)$--cable of $K \# K'$.
    Note that the $(n,1)$--cable of any knot (where $n\geq 2$) is prime by \cite[Theorem 4.4.1]{cromwellKnotsLinks2004a}.
    Then we have that $\alggenus(J_n)\neq 0$ because $\Delta_{K_n}(t) = \Delta_{T(n,1)}(t)\cdot\Delta_{K_1\#K_2}(t^n) \neq 1$ (by \cite[Theorem 6.15]{lickorishIntroductionKnotTheory1997} since $\alggenus(K_1\#K_2)\neq 0$).
    On the other hand, \cite{FMPC-Satelliteknots} 
    tells us how $\alggenus$ acts under satellite operations.
    In particular, $\alggenus(K_n) \leq \alggenus(T(n,1))+\alggenus(K_1\#K_2) = 1.$
    So we have that $\alggenus(K_n) = 1.$
    
    Also, by \cite[Theorem 2]{litherlandSignaturesIteratedTorus1979}, $\sigma_{K_n}(\omega) = \sigma_{T(n,1)}(\omega)+\sigma_{K_1\#K_2}(\omega^n)$.  Because $\sigma_{K_1\#K_2}(\omega)$ takes both positive and negative values, so does $\sigma_{K_n}(\omega) = \sigma_{K_1\#K_2}(\omega^n)$.  Proposition \ref{prop: H-slice sig} then implies that $\sd_a(K_n)>1.$
\end{proof}

\begin{figure}
    \includegraphics[width=0.3\textwidth]{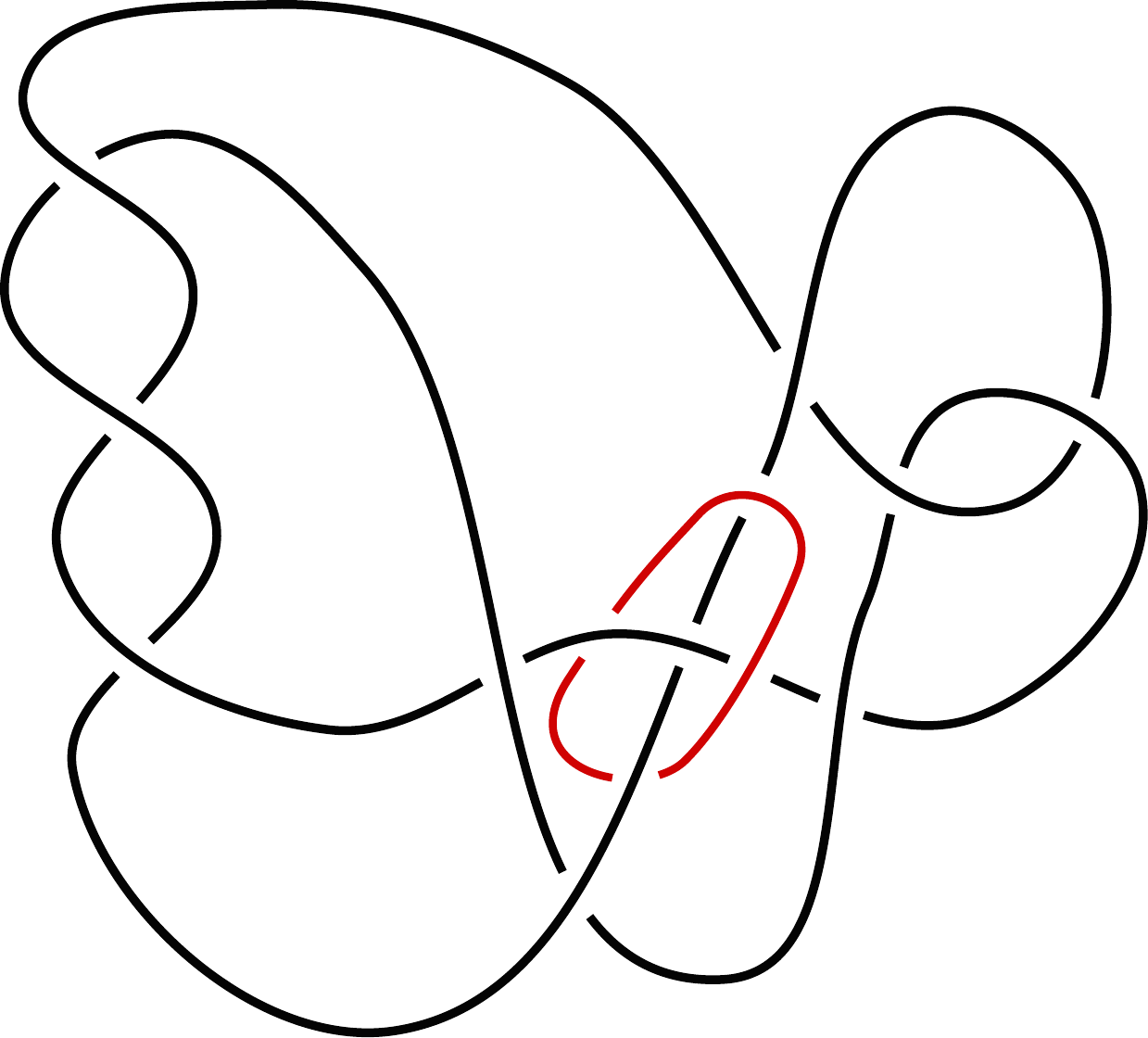} \quad \quad
    \includegraphics[width=0.3\textwidth]{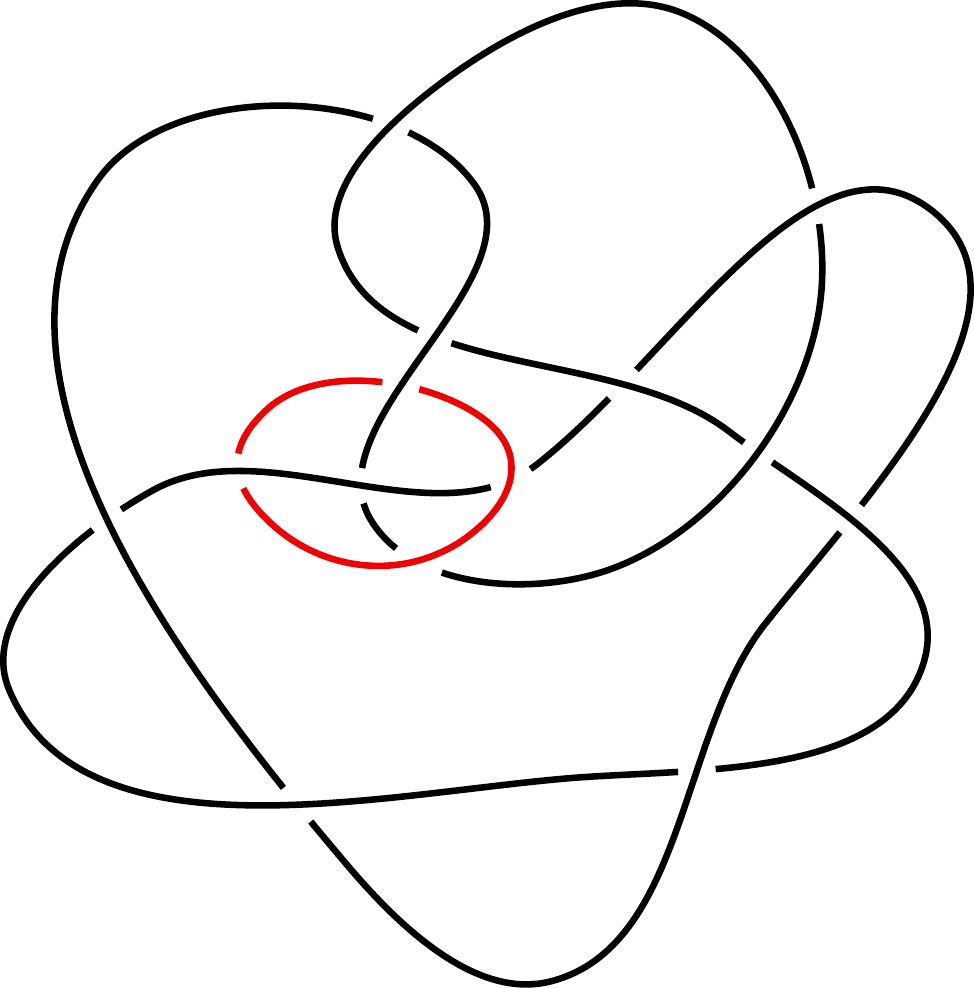}\\
    \vspace{10pt}
    \includegraphics[width=0.4\textwidth]{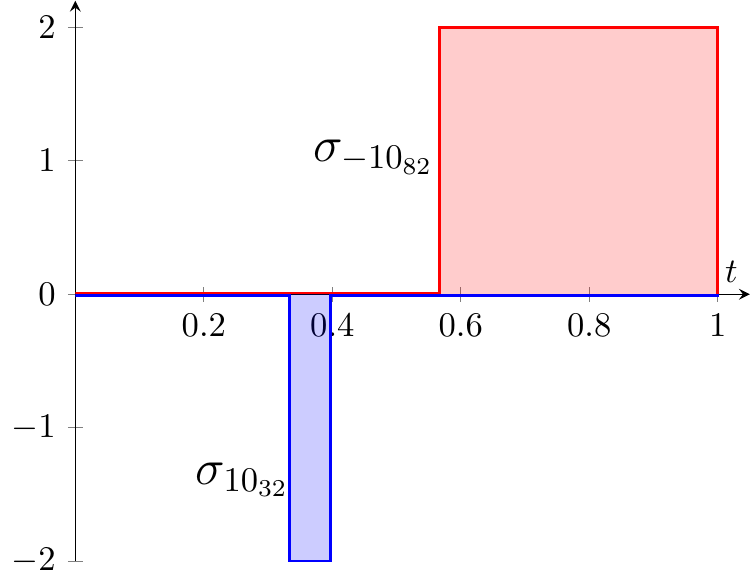}
    \caption{Top: the knots $10_{32}$ (left) and $10_{82}$ (right) with red unknots indicating an unknotting crossing change for each.  Bottom: the Levine-Tristram signature functions for $10_{32}$ and $-10_{82}$ (reparametrized so that $\omega = e^{2\pi i t}$). }
    \label{fig: sd_a>g_alg}
\end{figure}

We remark that, for any knot $K$ with $\alggenus(K)=1$ and $\sd_a(K)\geq 2$, inequality \ref{g_alg-chain} implies that $u_a(K)=2$ and hence that $\sd_a(K)=2$. A literature search suggests that $\{K_n\}$ is the first known infinite family of knots for which $\alggenus<u_a$. Note that, in \cite{feller_classical_2019}, the $3$-genus is used to distinguish between $\alggenus$ and $u_a$ for various knots since $\alggenus(K)\leq g_3(K)$ while $u_a \leq 2g_3(K)$. In our case, the $3$-genus of the $K_n$ grows large, and we use a different strategy for distinguishing between $\alggenus$ and $u_a$.

\section{Relationships between the surgery description and untwisting numbers}
\label{sec:sd-tu}

In the section above, we found an infinite family of knots for which $u_a = \tu_a = \sd_a = 2$. Other examples where $u_a = \tu_a = \sd_a$ are abundant.
We now endeavor to find examples where the two quantities (and other similar quantities) disagree.
In particular, in this section, we examine the square of inequalities below, and show that each inequality can be strict for infinitely many knots.
\begin{equation} 
    \label{ineqs}
    \begin{array}{ccc}
    \sd_a & \leq & \tu_a               \\
    \vertleq & & \vertleq \\
    \ & \ & \ \\
    \sd                    & \leq & \tu                
    \end{array}
\end{equation}

It is easy  to find infinitely many knots such that the vertical inequalities in equation \eqref{ineqs} are strict; for example, any nontrivial knots with Alexander polynomial $1$ satisfy $\sd_a < \sd$ and $\tu_a < \tu$.  Finding such examples for the horizontal inequalities in equation \eqref{ineqs} are more challenging. 

It is known that $u(K)$ and $\tu(K)$ can be arbitrarily different \cite{inceUntwistingInformationHeegaard2017a}.  In contrast, we can find no examples of knots in the literature with $\sd(K) \neq \tu(K)$. We provide the first known examples below; in fact, we find an infinite family $\{K_n\}$ of knots satisfying the stronger inequality $\sd(K_n) < \tu_a(K_n)$ for all $n\geq 2$.  This is the content of Theorem \ref{thm:sd != tu}. The same family provides infinitely many examples where $\sd_a < \tu_a = u_a$.

For the proof of Theorem \ref{thm:sd != tu},
we employ an obstruction to a knot having \emph{algebraic} unknotting number 1 due to Borodzik-Friedl \cite{borodzik_unknotting_2015}, which in turn generalizes an unknotting number $1$ obstruction due to Lickorish \cite{harper_unknotting_1985}. 
The obstruction involves the linking pairing on the first homology of the double-branched cover $\Sigma(K)$; see e.g., \cite{gordonAspectsClassicalKnot1978a} for a discussion of the linking pairing.

\begin{theorem}[\cite{borodzik_unknotting_2015}, Theorems 4.5 and 4.6] 
    \label{BF-theorem}
    If a knot $K$ can be \emph{algebraically} unknotted by a single crossing change, then there exists a generator $h$ of $H_1(\Sigma(K);\mathbb{Z})$ such that its linking pairing satisfies
    \[
        l(h,h)=\frac{\pm 2}{\det(K)}\in \mathbb{Q}/\mathbb{Z}
    \]
\end{theorem}

The proof of the following theorem follows Lickorish's proof that the knot $P(3,1,3)$ does not have unknotting number 1 (the main theorem of \cite{harper_unknotting_1985}).

\begingroup
\def\thetheorem{\ref{thm:sd != tu}}
\begin{theorem}
    There are infinitely many knots $\{K_n\}$ with $\sd(K_n)=1$ and $\tu(K_n)= \tu_a(K_n)=2$.
\end{theorem}
\addtocounter{theorem}{-1}
\endgroup

\begin{figure}
    \centering
    \def\svgwidth{\linewidth}
\begingroup%
  \makeatletter%
  \providecommand\color[2][]{%
    \errmessage{(Inkscape) Color is used for the text in Inkscape, but the package 'color.sty' is not loaded}%
    \renewcommand\color[2][]{}%
  }%
  \providecommand\transparent[1]{%
    \errmessage{(Inkscape) Transparency is used (non-zero) for the text in Inkscape, but the package 'transparent.sty' is not loaded}%
    \renewcommand\transparent[1]{}%
  }%
  \providecommand\rotatebox[2]{#2}%
  \newcommand*\fsize{\dimexpr\f@size pt\relax}%
  \newcommand*\lineheight[1]{\fontsize{\fsize}{#1\fsize}\selectfont}%
  \ifx\svgwidth\undefined%
    \setlength{\unitlength}{1063.18892832bp}%
    \ifx\svgscale\undefined%
      \relax%
    \else%
      \setlength{\unitlength}{\unitlength * \real{\svgscale}}%
    \fi%
  \else%
    \setlength{\unitlength}{\svgwidth}%
  \fi%
  \global\let\svgwidth\undefined%
  \global\let\svgscale\undefined%
  \makeatother%
  \begin{picture}(1,0.3227726)%
    \lineheight{1}%
    \setlength\tabcolsep{0pt}%
    \put(0,0){\includegraphics[width=\unitlength,page=1]{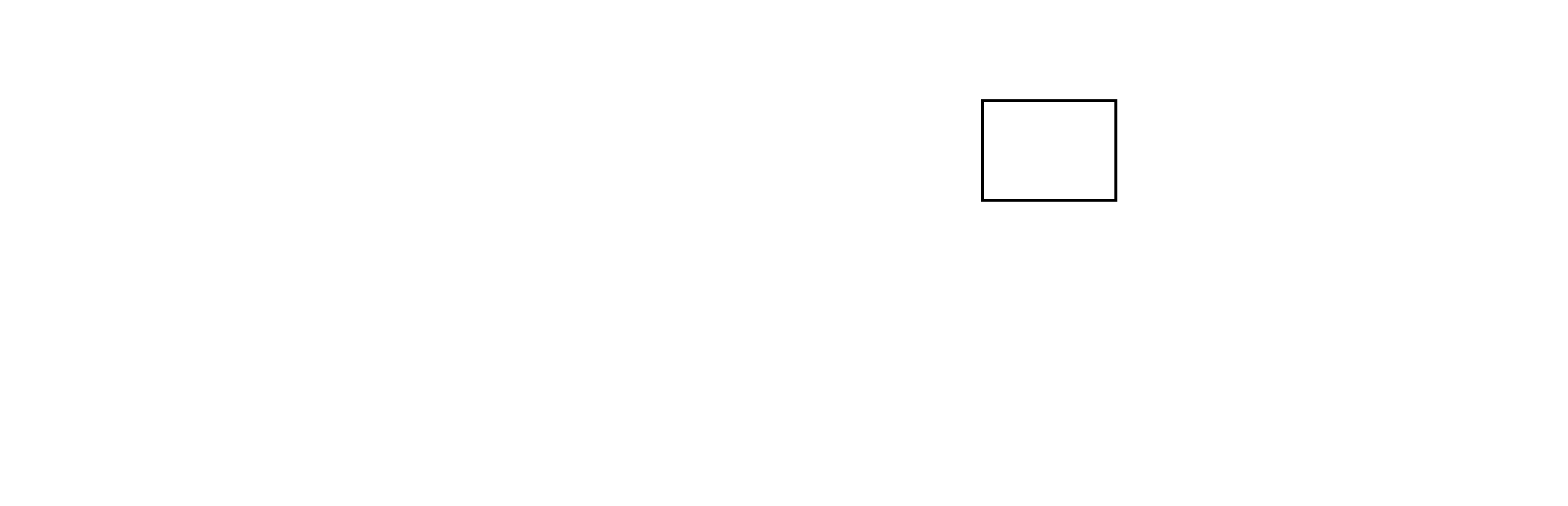}}%
    \put(0.004031,0.1626136){\color[rgb]{0,0,0}\makebox(0,0)[lt]{\lineheight{1.25}\smash{\begin{tabular}[t]{l}$10n+3$\end{tabular}}}}%
    \put(0.63240484,0.21921561){\color[rgb]{0,0,0}\makebox(0,0)[lt]{\lineheight{1.25}\smash{\begin{tabular}[t]{l}$10n+3$\end{tabular}}}}%
    \put(0,0){\includegraphics[width=\unitlength,page=2]{Pretzelknotsaretwo-bridge.pdf}}%
    \put(0.38907424,0.06137611){\color[rgb]{0,0.50196078,0.50196078}\makebox(0,0)[lt]{\lineheight{2.6500001}\smash{\begin{tabular}[t]{l}\textbf{$C$}\end{tabular}}}}%
  \end{picture}%
\endgroup%

    \caption{Two diagrams of the pretzel knots $K_n=P(10n+3,1,3)$; note that the boxed numbers represent \textit{half-}twists. On the left is the standard diagram for 3-strand pretzel knots, together with an unknotted curve $C$ where a $+1/2$ Dehn surgery can be applied to convert $K_n$ to the unknot. On the right is a diagram for the same knot in which it is more clear that the knots are two-bridge. In fact, they have Conway notation $C(10n+3,1,3)$.}
    \label{fig:pretzelknotsaretwobridge}
\end{figure}

\begin{proof}
    The family $K_n$ we consider is the set of pretzel knots of the form $P(10n+3,1,3)$; see Figure \ref{fig:pretzelknotsaretwobridge} for two isotopic diagrams and note that the boxed numbers represent \textit{half-}twists. We first note that $\sd(K_n)=1$ by performing the $+1/2$-surgery on the curve $C$ indicated in the figure. After the surgery, we obtain the pretzel knots $P(10k+3,1,-1)$, all of which are isotopic to the unknot.
    
    To conclude that $\tu(K_n)=2$, it is enough to show that $\tu(K_n)\neq 1$; $\tu(K_n)\leq 2$ since the surgery description move can be effected by two (single) null-homologous twists. 
    
    To show that $\tu(K_n)\neq 1$, first recall that $\tu_a(K)\leq \tu(K)$ and that $\tu_a(K)=u_a(K)$ for any knot $K$. We then assume that $\tu_a(K_n)=1$ for contradiction, and prove that the linking pairing on $H_1(\Sigma(K_n);\mathbb{Z})$ does not satisfy the condition in Theorem \ref{BF-theorem} for any $n\geq 1$. 
    
    First, note that the knots $K_n$ are 2-bridge; see Figure \ref{fig:pretzelknotsaretwobridge}. Each two-bridge knot has a (non-unique) associated fraction $p/q$ with the property that $\Sigma_2(K)\cong L(p,q)$; see, e.g., \cite[Chapter 2]{Kaw-Survey-Knot} for a discussion of two-bridge knots. In fact, $\{K_n\}$ are precisely those 2-bridge knots with continued fraction of the form 
    
    $$[10n+3,1,3]=10n+3+\frac{1}{1+\frac{1}{3}}=\frac{40n+15}{4}.$$
    
    Hence the double-branched covers of these knots $\Sigma(K_n)\cong L(40n+15,4)$ are lens spaces. So in particular, $\Sigma(K_n)$ can be obtained as surgery on a knot $J$ (in fact the unknot) via $\frac{40n+15}{4}$-surgery. This implies that $H_1(\Sigma(K_n);\mathbb{Z})$ is cyclic of order $40n+15$ generated by $\mu$ the image of a meridian of $J$ after surgery, and moreover that $l(\mu,\mu)=\frac{4}{40n+15}$ \cite{harper_unknotting_1985}.
    
    Any generator $h$ of $H_1(\Sigma(K_n))$ is of the form $h=t\mu$ for some integer $t$. Let $h$ be the generator which must exist according to Theorem \ref{BF-theorem} so that:
    \begin{equation}\label{equation:linkingpairing}
    \frac{\pm 2}{40n+15}=l(h,h)=l(t\mu,t\mu)=t^2\cdot l(\mu,\mu)=\frac{4t^2}{40n+15} \in \mathbb{Q}/\mathbb{Z}
    \end{equation}
    
    For the two fractions on the far left and far right of Equation \ref{equation:linkingpairing} to be equivalent in $\mathbb{Q}/\mathbb{Z}$, we must have $\pm 2\equiv 4t^2 \mbox{ (mod 40n+15)}$ so that $\pm 2$ must be a square $\mbox{(mod 40n+15)}.$ We will show that this is not true.
    
    If $\pm 2$ is a square $\mbox{(mod 40n+15)}$ then it also must be a square $\mbox{(mod a)}$ where $a$ is any factor of $40n+15$. In particular, $\pm 2$ must be a square $\mbox{(mod 5)}$. But neither $-2\equiv 3$ nor $2$ are squares $\mbox{(mod 5)}$. This is a contradiction, and hence $\tu_a(K_n)\neq 1$ for each $n$, which forces $\tu(K_n)\neq 1$. 
\end{proof}

Since $\sd_a(K)\leq \sd(K)$ for all knots $K$, the following corollary immediately follows.

\begin{corollary}
    \label{cor:infinite-sda-neq-tua} 
    There are infinitely many knots $\{K_n\}$ for which $\sd_a(K_n)=1$ while $\tu_a(K_n)=u_a(K_n)=2$.
\end{corollary}

\noindent Note that Corollary \ref{cor:infinite-sda-neq-tua} is the biggest gap we could hope for in the sense that $\sd_a(K)\leq \tu_a(K)=u_a(K)\leq 2\sd_a(K)$. 

While Theorem \ref{thm:sd != tu} provides infinitely many examples where $\sd < \tu_a$, one might ask if $\sd \leq \tu_a$ in general.  The following theorem provides an answer to this question in the negative. 

\begin{theorem}
    
    The $(p,1)$--cable of the untwisted Whitehead double of any nontrivial knot, which we denote $D_p$, has $\tu_a(D_p)=u_a (D_p)= 0<1=\sd(D_p)$ for all $p\in\mathbb{N}$.
\end{theorem}

\begin{proof}

    First, note that the Alexander polynomial of $D_p$, for any $p$, is equal to 1 (see the cabling relation in \cite{lickorishIntroductionKnotTheory1997}).  Thus $\tu_a(D_p) = 0$.  On the other hand, since $D_p$ is not unknotted, we must have that $\sd(D_p) \geq 1$. In fact, one can see that $\sd(D_p) = 1$ by performing a single 2-twist about the clasping region in the untwisted Whitehead double.
\end{proof}

\begin{note}
    It is possible to distinguish between $\sd$ and $\tu$ using obstructions from Heegaard Floer homology, though this seems feasible only to show that $\sd = 1 < 2 = \tu$ for individual knots. In particular, the $\sd$-moves in Figure \ref{fig:sd11a_103-10a_68} show that the knots $10_{68}$ and $11a_{103}$ have $\sd(K)=1$, though the fact that $\tu(10_{68}), \tu(11a_{103}) \geq 2$ is a result of the second author \cite[Theorem 1.2]{inceUntwistingInformationHeegaard2017a}.
\end{note}

\begin{figure}[ht!]
    \centering
    \subfloat{\includegraphics[scale = 0.5]{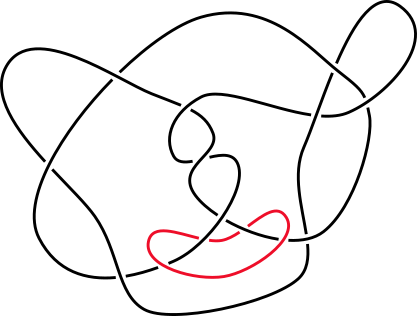}
    \label{fig:10a_68}}
    \quad\quad\quad
    \subfloat{
    \includegraphics[scale = 0.6]{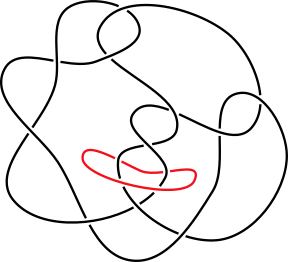}
    \label{fig:11a103:b}}
    \caption{
        \label{fig:sd11a_103-10a_68}
        The knots $10_{68}$ on the left and $11a_{103}$ on the right can be converted to the unknot by inserting two left (resp. right)-handed twists in the regions indicated by the red unknots.
    }
\end{figure}
 
For all examples we produce with $\sd\neq \tu$ the two invariants in fact only differ by 1. In the next section, we will prove Theorem \ref{thm:geometric} which states that the ratio of $tu$ to $sd$ is at most 3. This leaves open the following question.

\begin{question}\label{q::thm-sharp?} Does there exist a knot $K$ with $\sd(K)=1$ but $\tu(K)=3$, or in general so that $\tu(K) = 2\sd(K)+1$? 
\end{question}

Note that the techniques used in the proof of Theorem \ref{thm:sd != tu} cannot be used to obstruct a knot $K$ with $\sd(K)=1$ from having $\tu(K)\leq 2$ since the algebraic invariants can differ at most by a factor of 2.
It also seems unlikely that the Floer theoretic techniques of \cite{inceUntwistingInformationHeegaard2017a} would alone be enough to answer Question \ref{q::thm-sharp?} given the difficulty in obstructing knots from being $H$-slice in indefinite $4$-manifolds \cite{kjuchukova2021slicing}.
Thus, new techniques are likely needed to answer the question above.

\section{An inequality relating surgery description number and untwisting number} 
\label{sec:geometric}

In the previous section we asked whether a knot $K$ with $\sd(K)=1$ and $\tu(K)=3$ can exist, or more generally, if a knot with  $2\sd(K)+1=\tu(K)$ exists. 
In this section, we show that the untwisting number is at most twice the surgery description number plus $1$.

The following theorem was inspired by the work of Borodzik on algebraic $k$-simple knots \cite{borodzik_untwisting_2019-1}.  In addition, Duncan McCoy suggested the last portion of the proof of Theorem \ref{thm:geometric}, improving the upper bound from an earlier version of the paper.

\begingroup
\def\thetheorem{\ref{thm:geometric}}
\begin{theorem}
    For any knot $K$, we have that\,  $\sd(K) \leq \tu(K) \leq 2 \sd(K)+1.$
\end{theorem}
\addtocounter{theorem}{-1}
\endgroup

Note that while the following proof involves a series of Kirby calculus moves, the moves used are slam dunk moves (away from the knot), and handle slides involving only the added components (never the original knot); thus none of the moves alter the isotopy class of the knot.  The result is diagrammatic.

\begin{proof}
    The first inequality is clear from the definitions.  To show the second inequality, we will first show that an unknot of framing $\pm 1/(2k+1)$ which is null-homologous in the complement of $K$ can be replaced (via careful Kirby calculus) with two unlinked, null-homologous unknots, one with framing $+1$ and one with framing $-1$.  Thus $2k+1$ full twists in a single twisting region can be realized by a sequence of two full twists (of opposite signs) in some diagram of $K$.  This process (Procedure 1) is described below; an example in the case of five left-handed twists is shown in \Cref{fig:tu<=3sd}.  Throughout, we abuse notation and keep names of unknots unchanged after they have undergone a handle slide. 
    
    \noindent\textbf{Procedure 1:}
    \begin{enumerate}
        \item Use a reverse slam dunk move to view the $\pm 1/(2k+1)$--framed unknot as a $0$--framed unknot $U_1$ geometrically linked once with a $\mp (2k+1)$--framed unknot $U_2$ as in \Cref{fig:tu<=3sd:a}-\ref{fig:tu<=3sd:b}.
        \item By repeatedly sliding $U_2$ over $U_1$, one can reduce the framing on $U_2$ to $\pm 1$. See \Cref{fig:tu<=3sd:c}-\ref{fig:tu<=3sd:f}.  Note that, in each handle slide, only the portion of $U_2$ near $U_1$ is affected.
        While this changes how $K$ and $U_2$ are geometrically linked, the unknots $U_1$ and $U_2$ remain linked once. 
        \item Finally, slide $U_1$ over $U_2$.  This has the effect of changing the framing of $U_1$ by $\mp 1$.  See \Cref{fig:tu<=3sd:f}-\ref{fig:tu<=3sd:g}.
        After an isotopy (\Cref{fig:tu<=3sd:h}), it is not hard to see that the resulting $U_1$ and $U_2$ are unlinked.
    \end{enumerate} 
    
    We now show that unknots with framings $\pm 1$ and $\pm 1/(2k)$ which are null-homologous in the complement of $K$ can be replaced (again, via Kirby calculus) with three unlinked, null-homologous unknots, two with framings $\pm 1$ and one with framing $\mp 1$.  The process (Procedure 2) is described below; an example is shown in \Cref{fig:even}.

    \noindent \textbf{Procedure 2:}
     \begin{enumerate}
         \item Use a reverse slam dunk move to view the $\pm 1/(2k)$--framed unknot as a $0$--framed unknot $U_1$ geometrically linked once with a $\mp (2k)$--framed unknot $U_2$ as in \Cref{fig:even:a}-\ref{fig:even:b}.
         \item At the beginning of the procedure we assumed we had unknots with framings $\pm 1$ and $\pm 1/(2k)$. Call the unknot with $\pm 1$ framing $U_3$. Slide $U_2$ over $U_3$ with framing $\pm 1$ to change the framing on $U_2$ by 1. See \Cref{fig:even:c}.   At this stage, $U_2$ is linked with both $U_1$ and $U_3$. 
         \item Slide $U_3$ over $U_1$ in order to unlink $U_3$ from $U_2$.  The result is that, after an isotopy, $U_3$ is completely unlinked from $U_1$ and $U_2$.  In addition, $U_1$ and $U_2$ are in position to perform the procedure from the previous paragraph.  See \Cref{fig:even:d}-\ref{fig:even:e}.
         \item Apply steps (2) and (3) of Procedure 1.
         \end{enumerate}
         Thus, to see the upper bound, consider the following cases.  \begin{itemize}
             \item First, if the surgery description number can be realized using only $\pm(2k+1)$--moves (odd numbers of full twists in each twisting region), then we apply Procedure 1 to reduce each $(2k+1)$--move to a $+1$-- and $-1$--move.  Thus, in this case, $\tu(K) \leq 2\sd(K)$.
             \item Second, if at least one $\pm(2k)$--move (even number of full twists in a single twisting region) is required to realize the surgery description number, then replace one of the $\pm(2k)$--moves with parallel $\pm 1$-- and $\pm(2k-1)$--framed unknots.  Call the $\pm 1$--framed unknot $U_3$ and now use Procedure 2 with $U_3$ to reduce each remaining $\pm (2k)$--moves to a $+1$-- and $-1$--move.  Thus, $\tu(K) \leq 2\sd(K)+1$.
         \end{itemize}
\end{proof}
\begin{note}
In the proof of Theorem \ref{thm:geometric}, the upper bound of $2\sd(K)+1$ can only be sharp when \emph{every} minimal $\sd$-sequence for $K$ involves \emph{only} even numbers of full twists. In all other cases, consider a minimal $\sd$-sequence which involves at least one null-homologous $(2k+1)$-twist for some $k\in\mathbb{Z}$. We may use Procedure 1 on all $\pm 1/(2k+1)$-framed unknots to convert each into two $\pm 1$-framed unknots, then use Procedure 2 on all $\pm 1/(2k)$-framed unknots (if one exists) using one of the $\pm 1$-framed unknots obtained via Procedure 1 to build an untwisting sequence of length $2\sd(K)$.
\end{note}

\newpage

\begin{figure}[h]
    \centering

    \subfloat[Subfigure 1 list of figures text][Effecting null-homologous twist(s)]{\includegraphics[height=0.15 \textheight]{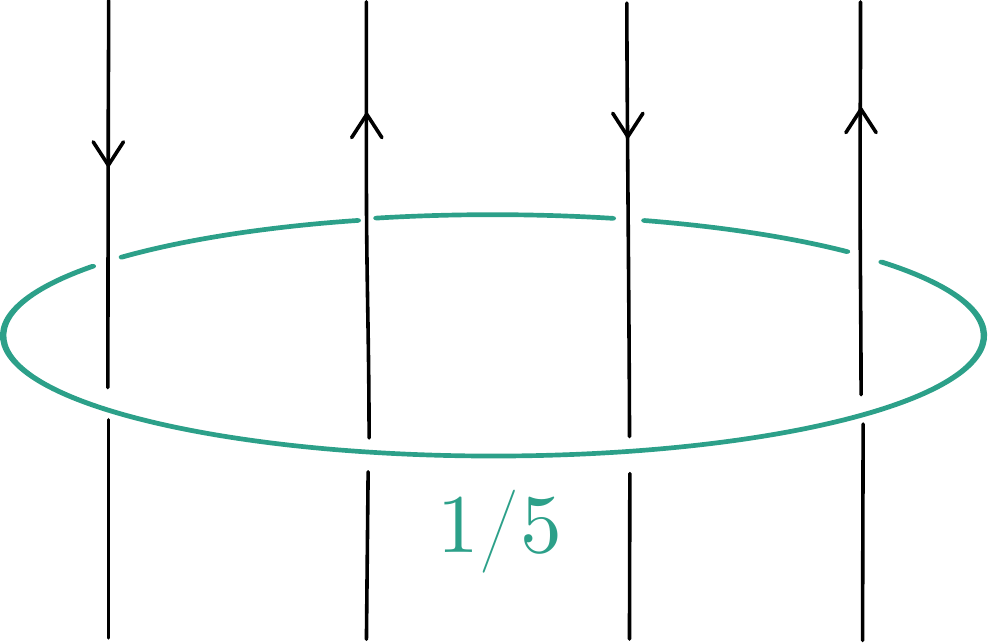}
    \label{fig:tu<=3sd:a}
    }
    \quad
    \subfloat[Subfigure 2 list of figures text][A slam dunk move]{
    \includegraphics[height=0.15\textheight]{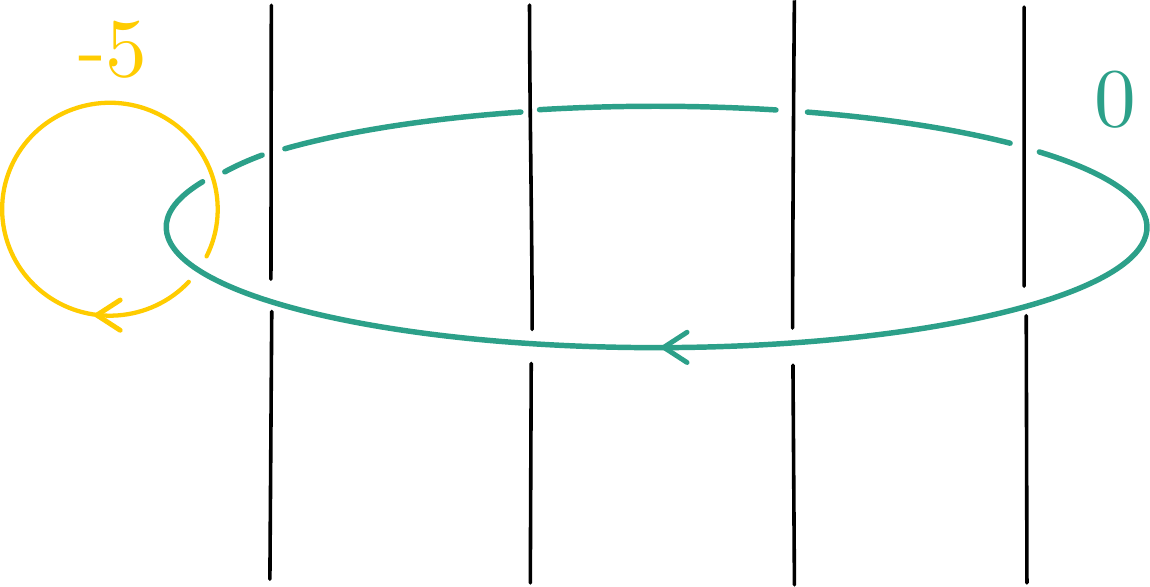}
    \label{fig:tu<=3sd:b}}
    
    \qquad
    \centering
    \subfloat[Subfigure 3 list of figures text][Handle addition]{
    \includegraphics[height=0.15\textheight]{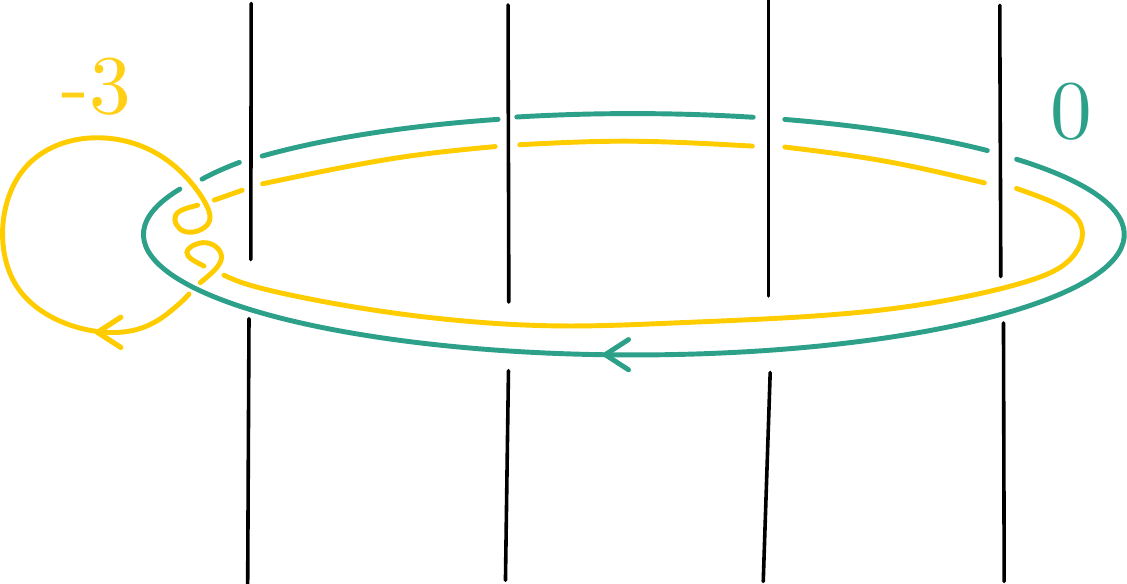}
    \label{fig:tu<=3sd:c}}
    \quad\quad\,\,
    \subfloat[Subfigure 4 list of figures text][Isotopy]{
    \includegraphics[height=0.15\textheight]{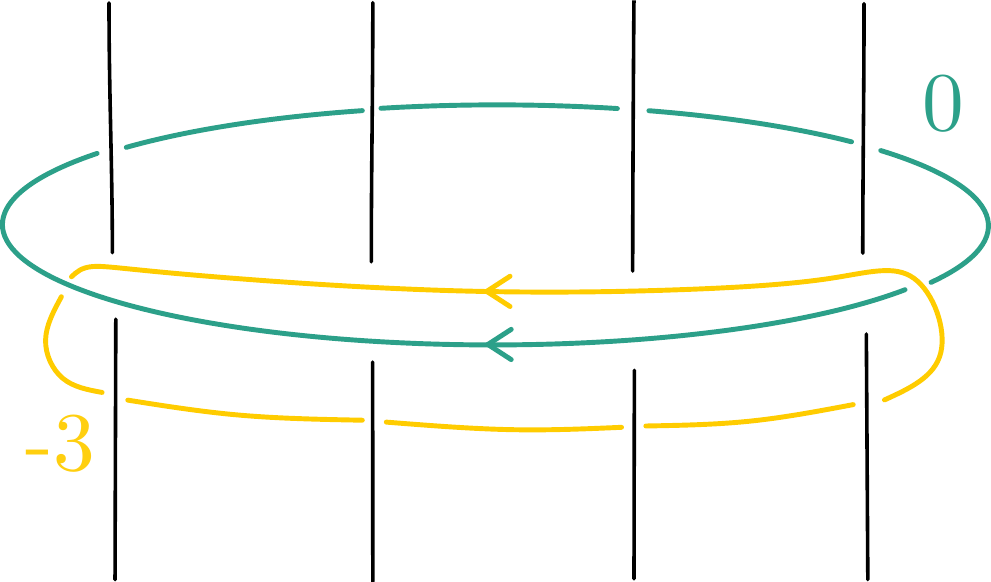}
    \label{fig:tu<=3sd:d}}

    \qquad
    \subfloat[Subfigure 5 list of figures text][Handle addition]{
    \includegraphics[height=0.17\textheight]{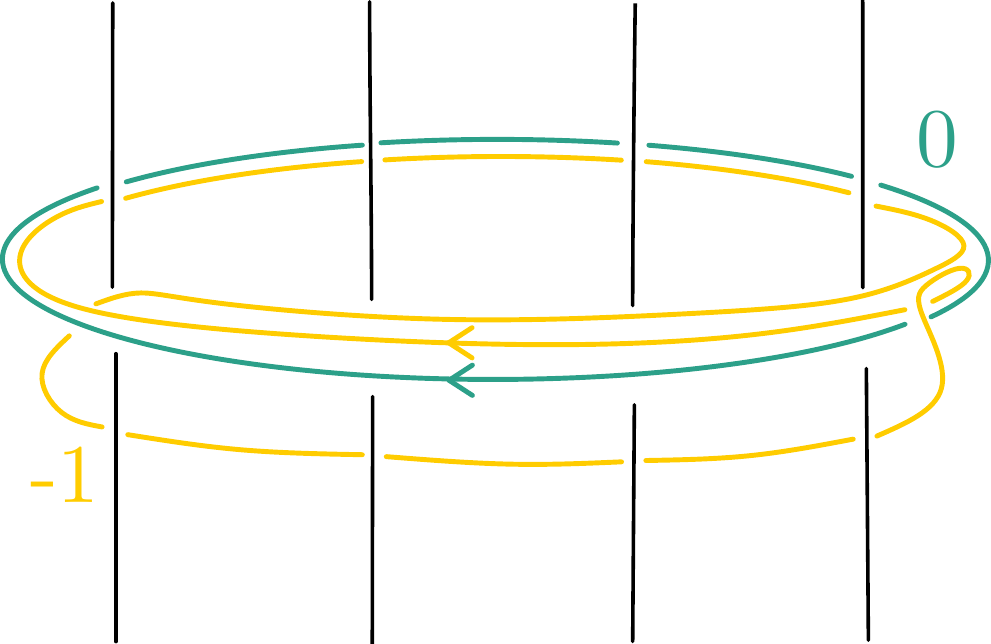}
    \label{fig:tu<=3sd:e}}
    \quad\quad
    \subfloat[Subfigure 6 list of figures text][Isotopy]{
    \includegraphics[height=0.17\textheight]{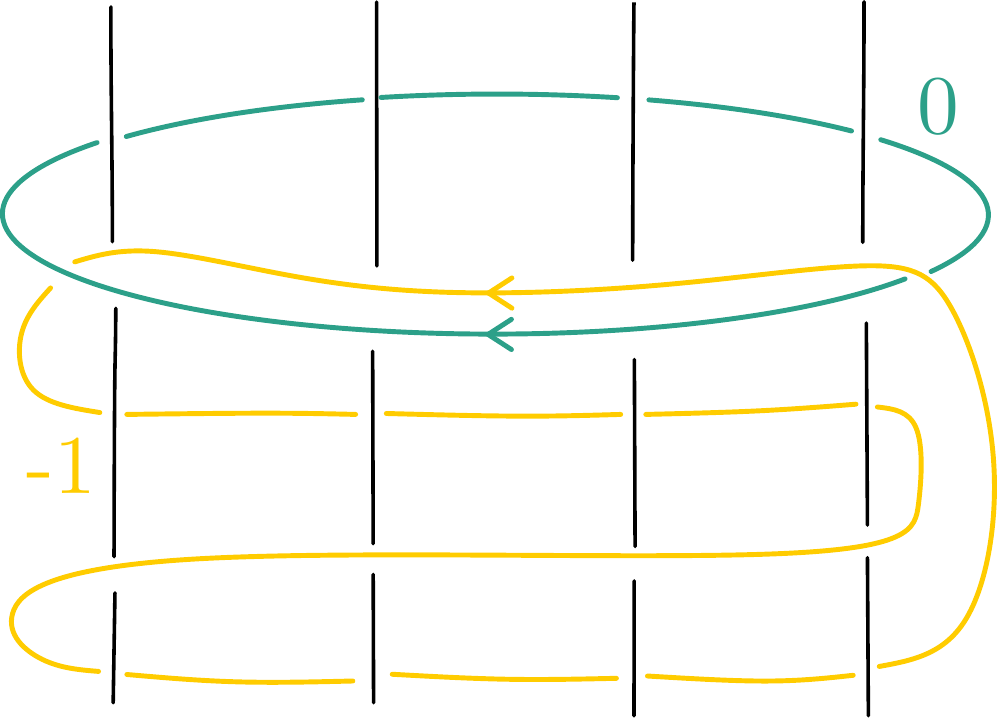}
    \label{fig:tu<=3sd:f}}
    
    \qquad
    \subfloat[Subfigure 7 list of figures text][Handle addition]{
    \includegraphics[height=0.2\textheight]{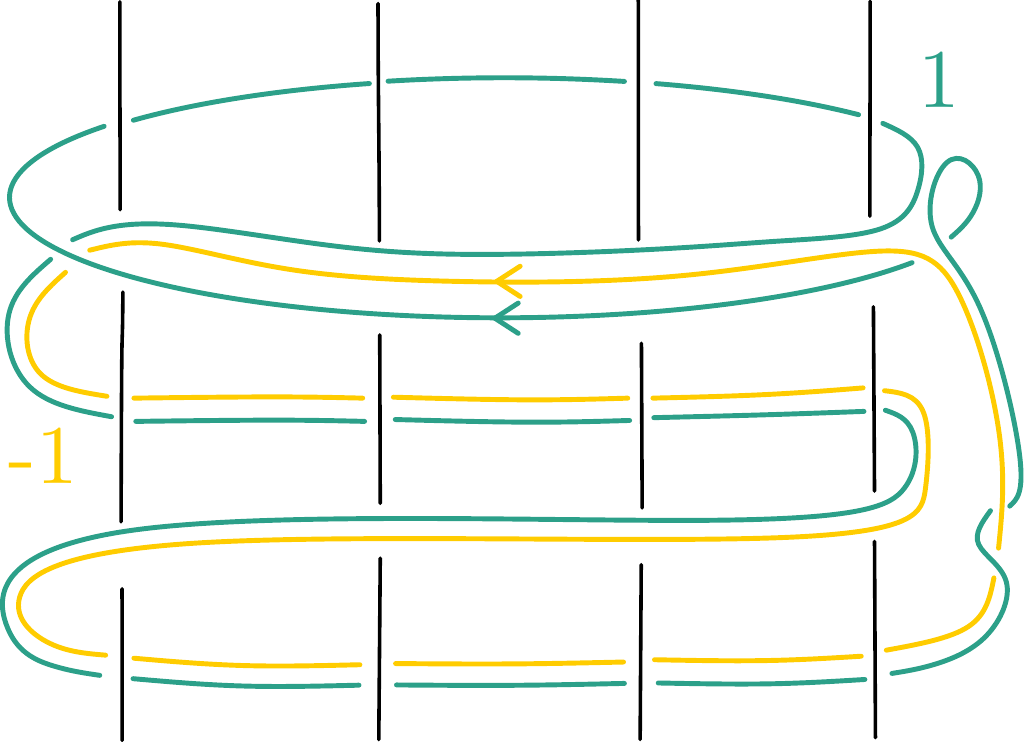}
    \label{fig:tu<=3sd:g}}
    \quad\quad
    \subfloat[Subfigure 8 list of figures text][Isotopy]{
    \includegraphics[height=0.2\textheight]{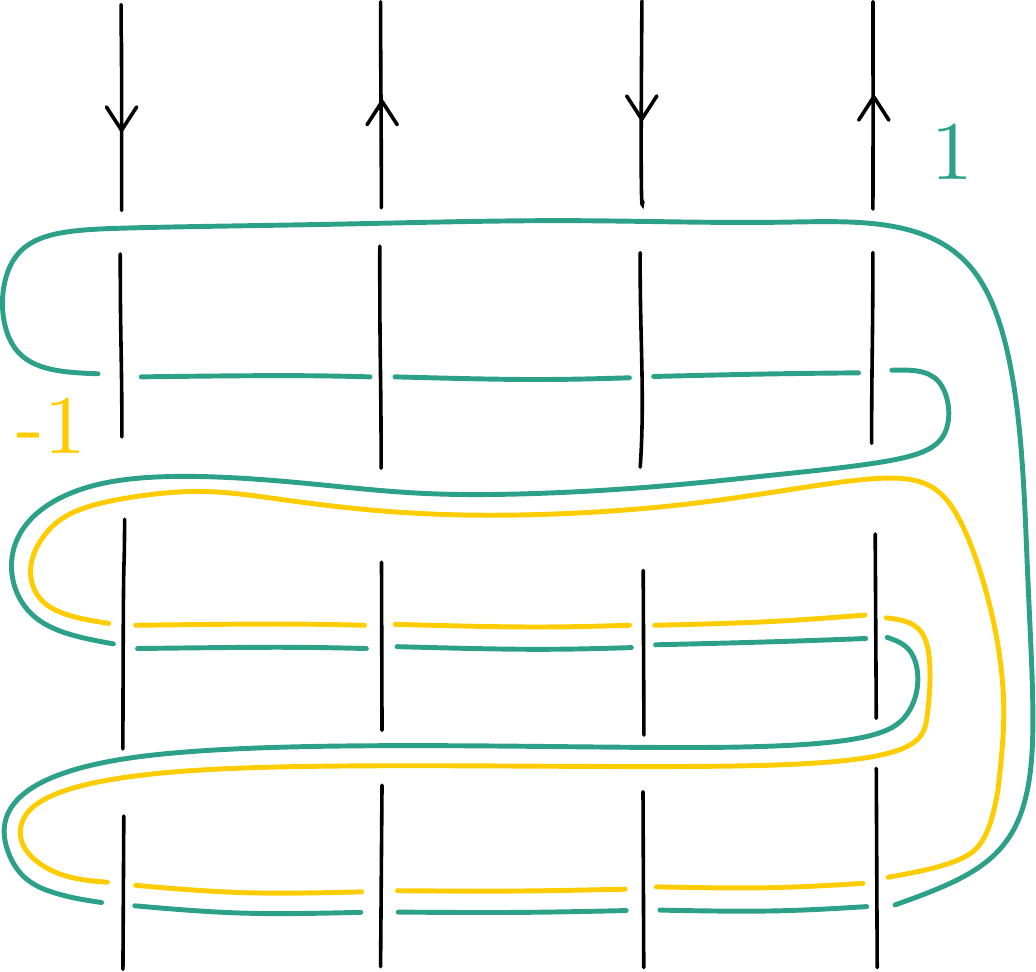}
    \label{fig:tu<=3sd:h}}
    \caption{A sequence of Kirby moves which shows that applying 5 parallel null-homologous twists can be obtained by two null-homologous twists}
    \label{fig:tu<=3sd}

\end{figure}
\newpage
\begin{figure}[ht]
    \centering

    \subfloat[Subfigure 1 list of figures text][\centering Null-homologous unknots in knot complement]  {\includegraphics[width=0.43 \textwidth]{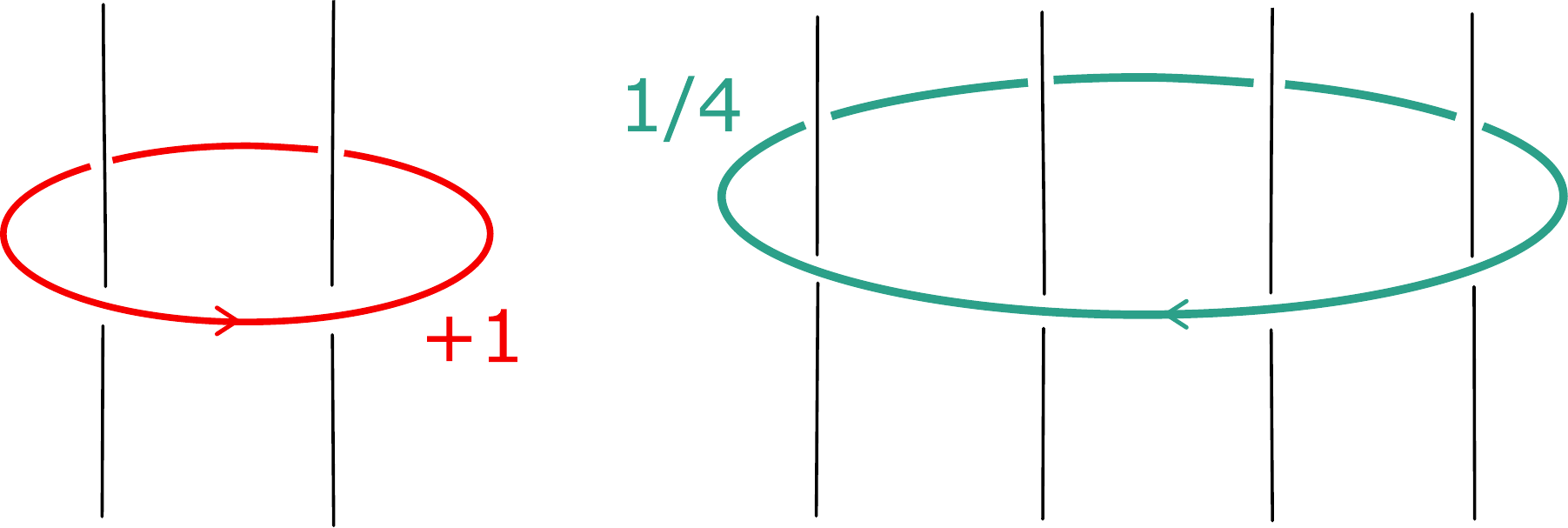}
    \label{fig:even:a}
    }
    \quad
    \subfloat[Subfigure 2 list of figures text][A slam dunk move]{
    \includegraphics[width=0.43\textwidth]{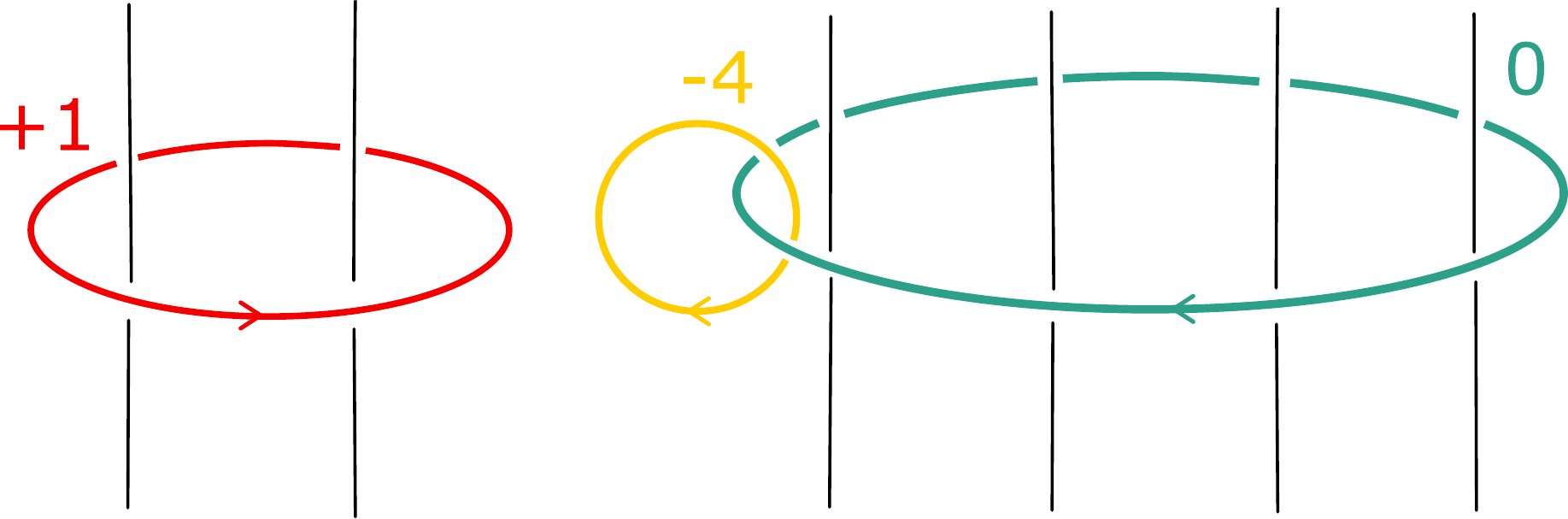}
    \label{fig:even:b}}
    
    \qquad
    \centering
    \subfloat[Subfigure 3 list of figures text][Handle addition]{
     \includegraphics[width=0.43\textwidth]{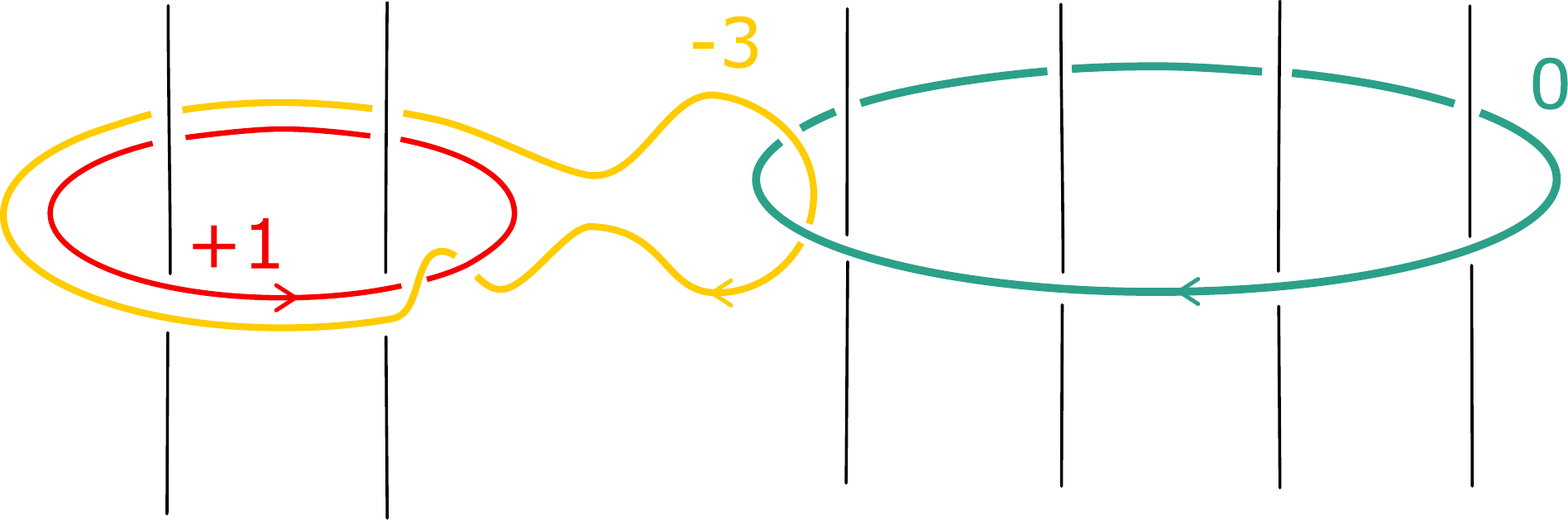}
    \label{fig:even:c}}
    \quad\quad\,\,
    \subfloat[Subfigure 4 list of figures text][Handle addition]{
     \includegraphics[width=0.43\textwidth]{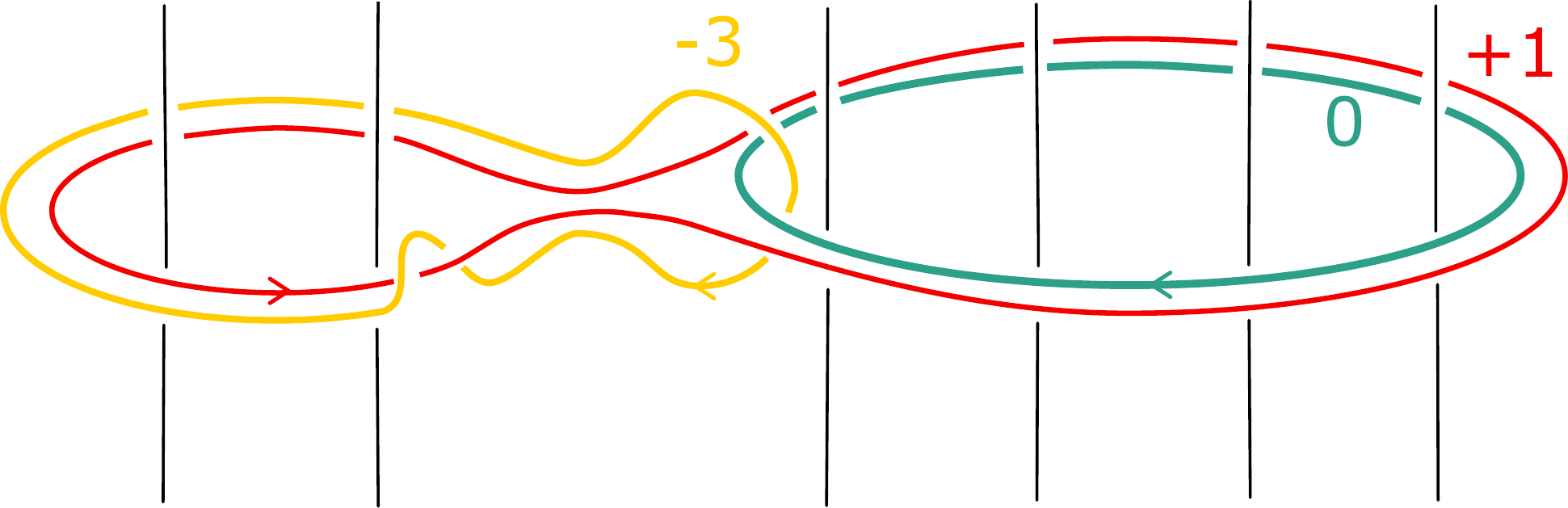}
    \label{fig:even:d}}

    \qquad
    \subfloat[Subfigure 5 list of figures text][Isotopy]{
     \includegraphics[width=0.43\textwidth]{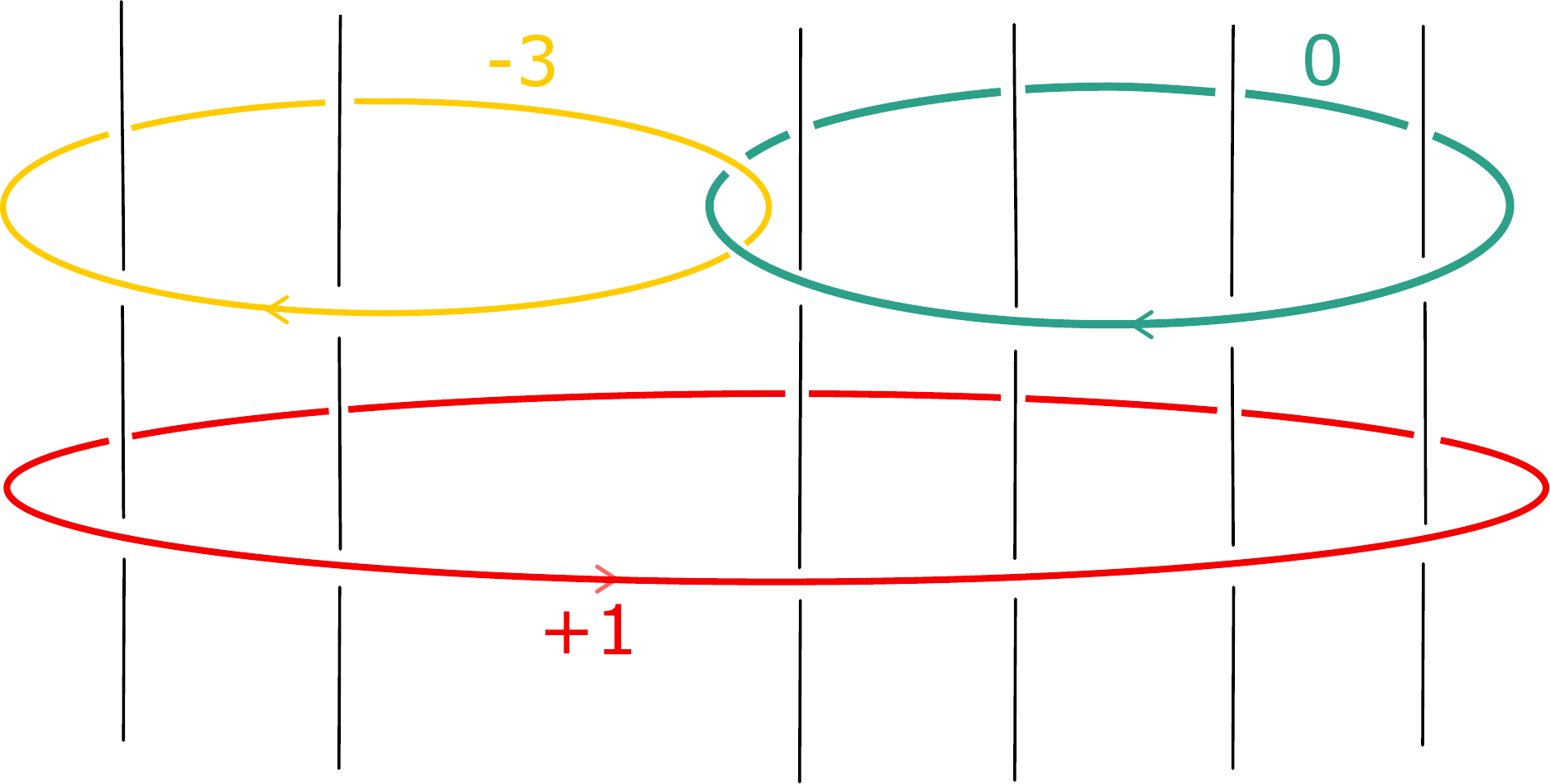}
    \label{fig:even:e}}
    \caption{A sequence of Kirby moves to replace $+1$-- and $+1/4$--framed null-homologous unknots in the knot complement with an unlinked $+1$--framed component and two components linked once, one with framing $0$.}
    \label{fig:even}

\end{figure}

\newpage
\printbibliography

@book {Kaw-Survey-Knot,
    AUTHOR = {Kawauchi, Akio},
     TITLE = {A survey of knot theory},
      NOTE = {Translated and revised from the 1990 Japanese original by the
              author},
 PUBLISHER = {Birkh\"{a}user Verlag, Basel},
      YEAR = {1996},
     PAGES = {xxii+420},
      ISBN = {3-7643-5124-1},
   MRCLASS = {57M25},
  MRNUMBER = {1417494},
MRREVIEWER = {Sergei K. Lando},
}

@article {FMPC-Satelliteknots,
    AUTHOR = {Feller, Peter and Miller, Allison N. and Pinz\'{o}n-Caicedo,
              Juanita},
     TITLE = {The topological slice genus of satellite knots},
   JOURNAL = {Algebr. Geom. Topol.},
  FJOURNAL = {Algebraic \& Geometric Topology},
    VOLUME = {22},
      YEAR = {2022},
    NUMBER = {2},
     PAGES = {709--738},
      ISSN = {1472-2747},
   MRCLASS = {57N70},
  MRNUMBER = {4464463},
       DOI = {10.2140/agt.2022.22.709},
       URL = {https://doi.org/10.2140/agt.2022.22.709},
}

@inbook{harper_unknotting_1985,
%   title = {The Unknotting Number of a Classical Knot},
%   booktitle = {Combinatorial {{Methods}} in {{Topology}} and {{Algebraic Geometry}}},
%   author = {Lickorish, W. B. R.},
%   date = {1985},
%   publisher = {{American Mathematical Soc.}},
%   abstract = {This collection marks the recent resurgence of interest in combinatorial methods, resulting from their deep and diverse applications both in topology and algebraic geometry. Nearly thirty mathematicians met at the University of Rochester in 1982 to survey several of the areas where combinatorial methods are proving especially fruitful: topology and combinatorial group theory, knot theory, 3-manifolds, homotopy theory and infinite dimensional topology, and four manifolds and algebraic surfaces. This material is accessible to advanced graduate students with a general course in algebraic topology along with some work in combinatorial group theory and geometric topology, as well as to established mathematicians with interests in these areas. For both student and professional mathematicians, the book provides practical suggestions for research directions still to be explored, as well as the aesthetic pleasures of seeing the interplay between algebra and topology which is characteristic of this field. In several areas the book contains the first general exposition published on the subject. In topology, for example, the editors have included M. Cohen, W. Metzler and K. Sauerman's article on ``Collapses of \$K\textbackslash times I\$ and group presentations'' and Metzler's ``On the Andrews-Curtis-Conjecture and related problems.'' In addition, J. M. Montesino has provided summary articles on both 3- and 4-manifolds.},
%   bookauthor = {Harper, John R. and Mandelbaum, Richard},
%   isbn = {978-0-8218-5039-8},
%   langid = {english},
%   keywords = {Mathematics / Topology}
% }

@book{cromwellKnotsLinks2004a,
  title = {Knots and {{Links}}},
  author = {Cromwell, Peter R.},
  year = {2004},
  month = oct,
  publisher = {{Cambridge University Press}},
  googlebooks = {djvbTNR2dCwC},
  isbn = {978-0-521-54831-1},
  langid = {english}
}

@book{lickorishIntroductionKnotTheory1997,
  title = {An {{Introduction}} to {{Knot Theory}}},
  author = {Lickorish, W. B. Raymond},
  year = {1997},
  month = oct,
  publisher = {{Springer Science \& Business Media}},
  googlebooks = {PhHhw\_kRvewC},
  isbn = {978-0-387-98254-0},
  langid = {english}
}

@incollection{litherlandSignaturesIteratedTorus1979,
  title = {Signatures of Iterated Torus Knots},
  booktitle = {Topology of {{Low-Dimensional Manifolds}}},
  author = {Litherland, R. A.},
  editor = {Fenn, Roger},
  year = {1979},
  volume = {722},
  pages = {71--84},
  publisher = {{Springer Berlin Heidelberg}},
  address = {{Berlin, Heidelberg}},
  doi = {10.1007/BFb0063191},
  isbn = {978-3-540-09506-4}, % 978-3-540-35186-3}

@article {feller_classical_2019,
    AUTHOR = {Feller, Peter and Lewark, Lukas},
     TITLE = {On classical upper bounds for slice genera},
   JOURNAL = {Selecta Math. (N.S.)},
  FJOURNAL = {Selecta Mathematica. New Series},
    VOLUME = {24},
      YEAR = {2018},
    NUMBER = {5},
     PAGES = {4885--4916},
      ISSN = {1022-1824},
   MRCLASS = {57M25 (57M27)},
  MRNUMBER = {3874707},
MRREVIEWER = {Sebastian Baader},
       DOI = {10.1007/s00029-018-0435-x},
       URL = {https://doi.org/10.1007/s00029-018-0435-x},
       eprint = {1611.02679},
    eprinttype = {arxiv},
  primaryclass = {math},
}

@article {livingston2018,
    AUTHOR = {Livingston, Charles},
     TITLE = {Signature functions of knots},
   JOURNAL = {Proc. Amer. Math. Soc.},
  FJOURNAL = {Proceedings of the American Mathematical Society},
    VOLUME = {146},
      YEAR = {2018},
    NUMBER = {10},
     PAGES = {4513--4520},
      ISSN = {0002-9939},
   MRCLASS = {57M25 (57M27)},
  MRNUMBER = {3834675},
MRREVIEWER = {Peter Feller},
       DOI = {10.1090/proc/14102},
       URL = {https://doi.org/10.1090/proc/14102},
}

@article {nakanishi2005,
    AUTHOR = {Nakanishi, Yasutaka},
     TITLE = {A note on unknotting number. {II}},
   JOURNAL = {J. Knot Theory Ramifications},
  FJOURNAL = {Journal of Knot Theory and its Ramifications},
    VOLUME = {14},
      YEAR = {2005},
    NUMBER = {1},
     PAGES = {3--8},
      ISSN = {0218-2165},
   MRCLASS = {57M25},
  MRNUMBER = {2124552},
MRREVIEWER = {G. Burde},
       DOI = {10.1142/S0218216505003701},
       URL = {https://doi.org/10.1142/S0218216505003701},
}

@incollection {mccoy2019nullhomologous,
    AUTHOR = {McCoy, Duncan},
     TITLE = {Null-homologous twisting and the algebraic genus},
 BOOKTITLE = {2019--20 {MATRIX} annals},
    SERIES = {MATRIX Book Ser.},
    VOLUME = {4},
     PAGES = {147--165},
 PUBLISHER = {Springer, Cham},
      YEAR = {2021},
   MRCLASS = {57K10},
  MRNUMBER = {4294766},
MRREVIEWER = {Bruno P. Zimmermann},
}

@article {MR4171377,
    AUTHOR = {Baader, S. and Banfield, I. and Lewark, L.},
     TITLE = {Untwisting 3-strand torus knots},
   JOURNAL = {Bull. Lond. Math. Soc.},
  FJOURNAL = {Bulletin of the London Mathematical Society},
    VOLUME = {52},
      YEAR = {2020},
    NUMBER = {3},
     PAGES = {429--436},
      ISSN = {0024-6093},
   MRCLASS = {57K10},
  MRNUMBER = {4171377},
       DOI = {10.1112/blms.12335},
       URL = {https://doi.org/10.1112/blms.12335},
    eprint = {1909.01003},
archivePrefix={arXiv},
primaryClass={math.GT}
}

@article {mccoy2021gaps,
    AUTHOR = {McCoy, Duncan},
     TITLE = {Gaps between consecutive untwisting numbers},
   JOURNAL = {Glasg. Math. J.},
  FJOURNAL = {Glasgow Mathematical Journal},
    VOLUME = {63},
      YEAR = {2021},
    NUMBER = {1},
     PAGES = {59--65},
      ISSN = {0017-0895},
   MRCLASS = {57K10},
  MRNUMBER = {4190070},
       DOI = {10.1017/S0017089520000014},
       URL = {https://doi.org/10.1017/S0017089520000014},
    eprint = {1908.06447},
archivePrefix={arXiv},
primaryClass={math.GT}
}

@inproceedings{gordonAspectsClassicalKnot1978a,
  title = {Some Aspects of Classical Knot Theory},
  booktitle = {Knot {{Theory}}},
  author = {Gordon, C. M.},
  editor = {Hausmann, Jean-Claude},
  year = {1978},
  series = {Lecture {{Notes}} in {{Mathematics}}},
  pages = {1--60},
  publisher = {{Springer}},
  address = {{Berlin, Heidelberg}},
  doi = {10.1007/BFb0062968},
  isbn = {978-3-540-35705-6},
  langid = {english}
}

@article{mathieu_chirurgies_1988-1,
	title = {Chirurgies de {Dehn} de pente $\pm 1$ sur certains n\ae{}uds dans les 3-vari{\'e}t{\'e}s},
	volume = {280},
	issn = {0025-5831, 1432-1807},
	url = {http://link.springer.com.ezproxy.rice.edu/article/10.1007/BF01456339},
	doi = {10.1007/BF01456339},
	number = {3},
	urldate = {2014-09-12},
	journal = {Mathematische Annalen},
	author = {Mathieu, Yves and Domergue, Michel},
	month = apr,
	year = {1988},
	pages = {501--508}
}

@article{kronheimer_gauge_1995,
	title = {Gauge theory for embedded surfaces, {II}},
	volume = {34},
	issn = {0040-9383},
	url = {http://www.sciencedirect.com/science/article/pii/0040938394E00033},
	doi = {10.1016/0040-9383(94)E0003-3},
	number = {1},
	urldate = {2014-08-26},
	journal = {Topology},
	author = {Kronheimer, Peter B. and Mrowka, Tomasz S.},
	month = jan,
	year = {1995},
	pages = {37--97}
}

@article{kauffman_signature_1976,
	title = {Signature of links},
	volume = {216},
	issn = {0002-9947, 1088-6850},
	url = {http://www.ams.org/tran/1976-216-00/S0002-9947-1976-0388373-0/},
	doi = {10.1090/S0002-9947-1976-0388373-0},
	urldate = {2015-08-05},
	journal = {Transactions of the American Mathematical Society},
	author = {Kauffman, Louis H. and Taylor, Laurence R.},
	year = {1976},
	pages = {351--365}
}

@article{murakami_algebraic_1990,
	title = {Algebraic unknotting operation},
	volume = {8},
	journal = {Proceedings of the Second Soviet-Japan Symposium of Topology},
	author = {Murakami, Hitoshi},
	year = {1990},
	pages = {283--292}
}

@article{kronheimer_gauge_1993,
	title = {Gauge theory for embedded surfaces, {I}},
	volume = {32},
	issn = {0040-9383},
	url = {http://www.sciencedirect.com/science/article/pii/004093839390051V},
	doi = {10.1016/0040-9383(93)90051-V},
	number = {4},
	urldate = {2014-08-26},
	journal = {Topology},
	author = {Kronheimer, Peter B. and Mrowka, Tomasz S.},
	month = oct,
	year = {1993},
	pages = {773--826}
}

@article{borodzik_unknotting_2015,
	title = {The unknotting number and classical invariants, {I}},
	volume = {15},
	issn = {1472-2739, 1472-2747},
	url = {http://msp.org/agt/2015/15-1/p04.xhtml},
	doi = {10.2140/agt.2015.15.85},
	number = {1},
	urldate = {2015-05-05},
	journal = {Algebraic \& Geometric Topology},
	author = {Borodzik, Maciej and Friedl, Stefan},
	month = mar,
	year = {2015},
	pages = {85--135}
}

@article{livingston_slicing_2002,
	title = {The slicing number of a knot},
	volume = {2},
	url = {http://www.emis.ams.org/journals/UW/agt/ftp/main/2002/agt-2s41.pdf},
	urldate = {2015-05-05},
	journal = {Algebr. Geom. Topol},
	author = {Livingston, Charles},
	year = {2002},
	pages = {1051--1060}
}

@article{borodzik_untwisting_2019-1,
  title = {Untwisting number and {{Blanchfield}} pairings},
  author = {Borodzik, Maciej},
  date = {2019-07},
  journaltitle = {Osaka Journal of Mathematics},
  volume = {56},
  number = {3},
  pages = {497--505},
  publisher = {{Osaka University and Osaka City University, Departments of Mathematics}},
  issn = {0030-6126},
  url = {https://projecteuclid.org/journals/osaka-journal-of-mathematics/volume-56/issue-3/Untwisting-number-and-Blanchfield-pairings/ojm/1563242421.full},
  urldate = {2021-09-01}
}

@article{ince_untwisting_2016,
  title = {The untwisting number of a knot},
  author = {\.{I}nce, Kenan},
  date = {2016-06-14},
  journaltitle = {Pacific Journal of Mathematics},
  volume = {283},
  number = {1},
  pages = {139--156},
  issn = {0030-8730, 0030-8730},
  doi = {10.2140/pjm.2016.283.139},
  url = {http://msp.org/pjm/2016/283-1/p08.xhtml},
  urldate = {2017-05-18},
  langid = {english}
}

@incollection {livingston_null-homologous_2019-1,
    AUTHOR = {Livingston, Charles},
     TITLE = {Null-homologous unknottings},
 BOOKTITLE = {Topology and geometry---a collection of essays dedicated to
              {V}ladimir {G}. {T}uraev},
    SERIES = {IRMA Lect. Math. Theor. Phys.},
    VOLUME = {33},
     PAGES = {59--68},
 PUBLISHER = {Eur. Math. Soc., Z\"{u}rich},
      YEAR = {2021},
   MRCLASS = {57K10},
  MRNUMBER = {4394501},
       DOI = {10.4171/IRMA/33-1/3},
       URL = {https://doi.org/10.4171/IRMA/33-1/3},
}

@misc{manolescu_zero_2021,
    title={From zero surgeries to candidates for exotic definite four-manifolds},
    author={Ciprian Manolescu and Lisa Piccirillo},
    year={2021},
    eprint={2102.04391},
    archivePrefix={arXiv},
    primaryClass={math.GT}
}

@article{inceUntwistingInformationHeegaard2017a,
  title = {Untwisting information from {{Heegaard Floer}} homology},
  author = {\.{I}nce, Kenan},
  year = {2017},
  month = aug,
  volume = {17},
  pages = {2283--2306},
  issn = {1472-2739, 1472-2747},
  doi = {10.2140/agt.2017.17.2283},
  archiveprefix = {arXiv},
  eprint = {1604.03033},
  journal = {Algebr. Geom. Topol.},
  number = {4}
}

@misc{kjuchukova2021slicing,
      title={Slicing knots in definite 4-manifolds}, 
      author={Alexandra Kjuchukova and Allison N. Miller and Arunima Ray and S\"{u}meyra Sakall\i},
      year={2021},
      eprint={2112.14596},
      archivePrefix={arXiv},
      primaryClass={math.GT}
}

@article {freedman1982,
    AUTHOR = {Freedman, Michael Hartley},
     TITLE = {The topology of four-dimensional manifolds},
   JOURNAL = {J. Differential Geometry},
  FJOURNAL = {Journal of Differential Geometry},
    VOLUME = {17},
      YEAR = {1982},
    NUMBER = {3},
     PAGES = {357--453},
      ISSN = {0022-040X},
   MRCLASS = {57N12 (57R80 57R99)},
  MRNUMBER = {679066},
MRREVIEWER = {John J. Walsh},
       URL = {http://projecteuclid.org/euclid.jdg/1214437136},
}

\end{document}